\documentclass[11pt, reqno]{amsart}

\usepackage{amssymb}
\usepackage{amsmath}
\usepackage{amscd}
\usepackage{float}
\usepackage{graphicx}
\usepackage{amsfonts}
\usepackage{color}
\newtheorem{conjecture}{Conjecture}
\newtheorem{theorem}{Theorem}
\newtheorem{corollary}{Corollary}
\newtheorem{lemma}{Lemma}
\newtheorem{proposition}{Proposition}
\newtheorem{definition}{Definition}
\newtheorem{remark}{Remark}
\newtheorem{examples}{Examples}
\newtheorem{example}{Example}

\newtheorem{notation}{Notation}

\def\B{\mathsf{B}}
\def\A{\mathsf{A}}
\def\M{\mathsf{M}}
\def\T{\mathsf{T}}

\def\E{\mathcal{E}}

\def\Supp{\rm{Supp}}
\begin{document}

\title{Markov trace on the algebra of braids and ties}

\author{Francesca Aicardi}
 \address{The Abdus Salam International Centre for Theoretical Physics (ICTP),  Strada  Costiera, 11,   34151  Trieste, Italy.}
 \email{faicardi@ictp.it}
 \author{Jes\'us Juyumaya}
 \address{Instituto de Matem\'{a}ticas, Universidad de Valpara\'{i}so,
 Gran Breta\~{n}a 1111, Valpara\'{i}so, Chile.}
 \email{juyumaya@gmail.com}

\thanks{
This work was written   during a Visiting Fellowship  of the second author at the ICTP, Trieste (Italy).
The second author has been supported partially by Fondecyt 1141254,   DIUV  No 1--2011 and by
the European Union (European Social Fund  ESF) and Greek national funds through the Operational Program
  \lq Education and Lifelong Learning\rq\ of the National Strategic Reference Framework (NSRF)  Research Funding
Program: THALES: Reinforcement of the interdisciplinary and/or interinstitutional research and innovation.
}


\date{}

\begin{abstract}
We prove  that the so--called {\it algebra of braids and ties}
supports a Markov trace. Further, by using this trace in the  {\it Jones' recipe},
we define    invariant polynomials  for classical knots and singular knots. Our invariants      have three parameters. The invariant
of classical knots is an extension of the   Homflypt polynomial
and the invariant of singular knots is an extension   of  an
invariant of singular knots previously  defined by  S. Lambropoulou and  the second author.
\end{abstract}

\maketitle

\section{Introduction}

The algebra of braids and ties (defined by generators and relations) firstly appeared in \cite{juICTP},
having the purpose  of constructing  new representations of  the braid group.
Later, the  first  author observed   that  the definition had  a redundant relation and  provided a graphical interpretation of the generators and relations in terms of  braids and ties. In \cite{aijuICTP1} we have
investigated this algebra, proving in particular  that  it is finite--dimensional
and discussing the representation theory in low dimension.

\smallbreak

Let $n$ be a positive integer. The   algebra of braids   and ties with parameter $u$ is
denoted $\mathcal{E}_n(u)$.  Its generators
can be regarded as elements of the Yokonuma--Hecke algebra ${\rm Y}_{d,n}(u)$\cite{juJKTR}. In fact, the    defining relations of
$\mathcal{E}_n(u)$
 come  out by imposing the commutation relations of the braid generators  of  ${\rm Y}_{d,n}(u)$  with certain idempotents in
 ${\rm Y}_{d,n}(u)$ appearing  in the square of the braid generators,   for details see Subsection 3.2.

\smallbreak

The algebra $\mathcal{E}_n(u)$ was studied  by S. Ryom--Hansen in \cite{rhJAC}.
He constructed a faithful  tensorial representation (Jimbo--type) of this algebra
which is used to classify the irreducible representations
 of $\mathcal{E}_n(u)$. Notably, he constructed a basis, showing that the dimension of the algebra is $b_nn!$, where
 $b_n$ denotes the $n$--th Bell number. This basis plays a crucial role here to prove that $\mathcal{E}_n(u)$
 supports a Markov trace. Likewise, the algebra was considered by E. Banjo in  her Ph. D. thesis, see
\cite{baART}. She has related $\mathcal{E}_n(u)$ to the ramified partition algebra \cite{maelMZ}. More precisely,
E. Banjo  has shown an explicit isomorphism between  the  specialized  algebra  $\mathcal{E}_n(1)$ and a small ramified
partition algebra;  by using this isomorphism she determined the complex generic representation of $\mathcal{E}_n(u)$.

\smallbreak

Looking  at   the graphical interpretation of the generators of $\mathcal{E}_n(u)$    given in \cite{aijuICTP1}, it is natural to try
to define an invariant of knots through the same mechanism (Jones' recipe) that  defines the famous Homflypt
polynomial \cite{joAM}. To do that it is essential to have a Markov trace on  $\mathcal{E}_n(u)$.
 Since the  algebra  $\mathcal{E}_n(u)$   was provided  with  a  basis  by Ryom--Hansen,   a first
 attempt was to define a trace by  the  same inductive method used  to define the Ocneanu trace on the
 Hecke algebras, i.e., by constructing an  isomorphism between    the  algebra at level $n$  and
 a direct sum  of algebras   at lower levels, for details see the proof  \cite[Theorem 5.1]{joAM}.
 Unfortunately, we could not  reproduce  this method in our situation because  the  Ryom--Hansen basis
 cannot be  defined -- at  least simply -- in an inductive manner. We have then adopted  successfully
 the method of {\it relative traces}  \cite{chpoArxiv, ogpo}, using  as main reference
 the work of  M. Chlouveraki and L. Poulain d'Andecy  \cite[Section 5]{chpoArxiv},
 where it is proved that certain affine and cyclotomic Yokonuma--Hecke algebras support a Markov trace.
 Other  works where  the method of relative traces appears are \cite{ogpo, isog1, isog2, iski}, but  we don't know whom to  credit   for this method.

\smallbreak

  In this paper we prove that  $\mathcal{E}_n(u)$ supports a Markov trace $\rho$, that  depends on two parameters
 $\mathsf{A}$ and $\mathsf{B}$. Then, by using as ingredient
$\rho$ in the  Jones' recipe \cite{joAM} and a representation of the braid group (respectively, of the braid monoid)
in $\mathcal{E}_n(u)$,   we define an invariant, $\bar{\Delta}(u,\A,\B)$,  for classical knots
(respectively, $\bar{\Gamma}(u,\A,\B)$, for singular knots),
 with parameters $u$, $\mathsf{A}$ and $\mathsf{B}$. Since the definitions of these invariants essentially  uses   the same formula
 given by Jones to define the Homflypt polynomial, we can see that the specialization $\bar{\Delta}(u, \mathsf{A}, 1)$
 is in fact the Homflypt polynomial. Also, for the same reason it is clear that $\bar{\Delta}(u, \mathsf{A}, 1/m )$ (respectively
 $\bar{\Gamma}(u, \mathsf{A}, 1/m )$), where $m$ is a positive integer, coincides with the invariant of classical knots  (respectively singular knots),  defined by
  S. Lambropoulou and  the second author  in \cite{julaMKTA} (respectively \cite{julaJKTR}).   For more information   on how strong is the  invariant $\bar{\Delta}$ ,
see the Addendum at the end of the paper.

 \smallbreak

 Finally, we  note that the invariants defined here
can be recovered   from an  invariant  for tied  knots, see \cite{AJlinks}. The tied knots constitute  in fact  a  new   class of knot--like objects
  in the Euclidean space whose definition  is motivated   by the graphical interpretation   of the defining generators
of $\mathcal{E}_n(u)$   in terms of    braid and ties, given in Section \ref{sectiondiagram}.

\smallbreak

 The structure of the paper is as follows. In Section \ref{sectionnotations}, we give the
necessary   background  and  the  notation used in the paper. In Section \ref{sectionthealgebraof},  we recall
the definition and the origin of the algebra of braids and ties (Subsections 3.1 and 3.2 respectively). In Subsections 3.3 and 3.4,
 we collect some  further relations  coming from  the defining relations of    $\mathcal{E}_n(u)$ and  we recall the definition  of  a  basis, found by S. Ryom--Hansen, for this algebra (\cite{rhJAC}). In addition, we show some relations among  the elements of  Ryom--Hansen basis which will be used later.
  Section \ref{sectionmarkovtrace} contains the main result of this paper, where  we prove     that  $\mathcal{E}_n(u)$ supports a Markov trace. This Section is divided in  two subsections. The  first  one  is devoted to the construction of a family of relative traces
(Theorem \ref{relativet})
which are used in the  second  Subsection for the definition  of the Markov trace on    $\mathcal{E}_n(u)$, see Theorem \ref{trace}. In Section \ref{sectionapplications},
 we construct an  invariant of classical links  (Theorem \ref{Delta})  and an invariant of singular knots  (Theorem \ref{Gamma}). These  invariants  can be interpreted, respectively, as a
 generalization of the Homflypt polynomial and   as a generalization of the invariant defined by
 S. Lambropoulou and the second author;   in Subsection 5.3, we  explain     how these invariants are related.
   In Section \ref{sectiondiagram}, we recall the diagrammatic interpretation of the defining generators of    $\mathcal{E}_n(u)$ given in \cite{aijuICTP1}   according to  which  one part of the generators is represented as usual braid and the other part as ties. Moreover,  we generalize  the  diagrammatic interpretation  of  the  ties  generators,  so  that  the elements of  the basis by  Ryom--Hansen result in  a   quite  simple     form  and some defining  relations  of  the  algebra  become then evident.  The paper ends with Section 7, dealing with an interesting question posed by the referee.  Also,  in this   section we introduce   the idea of
 \lq deframization\rq\ of a framized algebra,    and  we propose  a natural
deframization of  certain framized algebras defined by M. Chlouveraki and L.
Poulain d'Andecy in \cite{chpoArxiv}.

{\bf Acknowledgements.} The  authors  are  deeply  grateful  to  an  anonymous  referee,  who  pointed  out  different  gaps  and  inaccuracies  in  the original  manuscript, making useful  remarks  and  posing  interesting  questions. A  particular  thank also  to  Marcelo Flores, for having noticed  the  redundancy  of  a lemma  in  a previous  version.

\section{Notations and background}\label{sectionnotations}

\subsection{}
Let $u$ be an indeterminate. We denote by $K$  the field of the rational functions $\mathbb{C}(u)$.

The term algebra here indicates an associative unital (with unity $1$)  algebra over  $K$. Thus,   $K$  can  be  viewed as a subalgebra of the center of the algebra.

As usual we denote by $B_n$ the braid group on $n$ strands. Thus, $B_{n}$ has the Artin presentation by
generators $\sigma_1, \ldots ,\sigma_{n-1}$  and the {\it braid relations}:
$\sigma_{i}\sigma_{j} = \sigma_{j}\sigma_{i}$, for $\vert i-j \vert >1$ and
$\sigma_{i}\sigma_{i+1}\sigma_{i}=\sigma_{i+1}\sigma_{i}\sigma_{i+1}$, for $i\in\{1, \ldots , n-2\}$. We
assume   that  the braid generators $\sigma_{i}$ have positive crossing, represented by the following diagram:

  \begin{figure}[h]
  \centering
  \includegraphics* {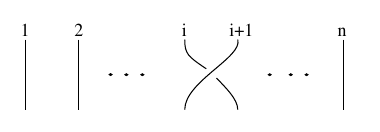}
 \caption{ The braid generator $\sigma_i$}\label{Fig1}
  \end{figure}

Let $S_n$ be the symmetric group on   the  alphabet {\bf n}$=\{1,\ldots,n\}$, and $s_i$ the transposition
$(i,\, i+1)$. Recall that every element $w\in S_n$ can be written (uniquely) in the   following  form
\begin{equation}\label{wcanonico}
w= w_1w_2\cdots w_{n-1}
\end{equation}
where $w_i\in \{1, s_i, s_is_{i-1},\ldots , s_is_{i-1}\cdots s_1 \}$.

\subsection{}
We denote by    $\mathsf{P}({\bf n})$ the  set formed by the set--partitions of $\bf n$.
The cardinality of $\mathsf{P}({\bf n})$ is called the $n$--th Bell number.

The pair  $(\mathsf{P}({\bf n}), \preceq)$ is  a  poset.  More precisely, setting  $I:=(I_1, \ldots , I_r)$, $J := (J_1, \ldots , J_s)
 \in \mathsf{P}({\bf n})$,   the partial order $\preceq$ is defined by
 $$
 I\preceq J \quad \text{if and only if each  $J_k$  is a union of some $I_m$'s}.
 $$
If $I\preceq J$ we  shall  say  that  $I$    {\sl refines}  $J$.

 The  subsets of  $\bf{n}$ entering a  partition   are  called blocks.
  For  short   we shall omit the subsets  of cardinality $1$ (single  blocks) in the partition. For example, the  partition $I=(\{1,3\}, \{2\}, \{4\}, \{5\}, \{6\} )$ in $\mathsf{P}({\bf 6})$, will be simply written as  $I=(\{1,3\})$.   Moreover,   Supp$(I)$  will  denote  the  union  of  non-single  blocks of $I$.

 \smallbreak

The symmetric group $S_n$ acts naturally on $\mathsf{P}({\bf n})$. More precisely, set $I=(I_1, \ldots , I_m)\in \mathsf{P}({\bf n})$.
The action $w(I)$  of $w \in S_n$ on $I$  is given by
\begin{equation}\label{wsym}
w(I) := ( w(I_1), \ldots , w (I_m))
\end{equation}
where $w(I_k)$ is the subset of $\bf n$ obtained  by applying  $w$ to  the set $I_k$.

If $I$  and  $J$  are two  set--partitions in $\mathsf{P}({\bf n})$,  we  denote  $  I \ast J $  the  minimal  set--partition   refined by   $I$ and  $J$. Let $L$ be a  subset of  {\bf{n}}. During  the  work we  will use for  short  $I \ast L$ to  indicate  $I \ast (L)$.  So,    $  I \ast \{j,m\} $  coincides  with  $I$  if  $j$  and  $m$  already belong  to  the  same   block in  $I$, otherwise, $ I \ast \{j,m\} $   coincides  with  $I$  except for   the two  blocks  containing  $j$  and  $m$,  that  merge  in  a  sole  block.  For short,  we shall denote  by  $I \ast j $ the  set--partition $ I\ast \{j,j+1\} $. For instance, for the set--partition $I= (\{1,2, 4\}, \{3,5,6\})$ :
 $$
 I\ast \{1,4\} =I \quad\text{and}\quad
 I \ast 2 =(\{1,2,3,4,5,6\}).
 $$

Finally, for $I\in \mathsf{P}({\bf n})$, we denote   $I\backslash n$  the element in $\mathsf{P}({\bf n-1})$ that is obtained
 by removing $n$ from $I$. For example, for the set--partition $I$ of  the  example  above, $I\backslash 6 =  (\{1,2, 4\}, \{3,5\})$.


\section{The algebra of braids and ties}\label{sectionthealgebraof}

\subsection{}    We recall here the definition of the algebra of braids and ties ${\mathcal E}_n(u)$.   From now on, for brevity, we call this algebra simply the bt--algebra. Further,  we shall omit $u$ in  ${\mathcal E}_n(u)$.

\begin{definition}\rm
  We set ${\mathcal E}_1 = K$ and  for  every   integer   $n >1$  we  define ${\mathcal E}_n$ as  the algebra generated by $ T_1, \ldots, T_{n-1}, E_1, \ldots , E_{n-1}$
 satisfying the following relations:
\begin{eqnarray}
\label{E1}
T_iT_j& = & T_jT_i  \qquad \text{ for all $i,j$ such that   $\vert i-j\vert >1$} \\
\label{E2}
T_iT_jT_i& = & T_jT_i T_j \qquad \text{ for all $i,j$ such that  $\vert i-j\vert = 1$} \\
\label{E3}
T_i^2 & = & 1 + (u-1) E_{i}  \left( 1+  T_i\right)\qquad \text{ for all $i$} \\
\label{E4}
E_iE_j & =  &  E_j E_i \qquad \text{ for all $i,j$}\\
\label{E5}
E_i^2 & = &  E_i \qquad \text{ for all $i$}  \\
\label{E6}
E_iT_i  & = &   T_i E_i \qquad  \text{ for all  $i$ } \\
\label{E7}
E_iT_j  & = &   T_j E_i \qquad  \text{ for all  $i,j$  such that  $\vert i - j\vert > 1$}\\
\label{E8}
E_iE_jT_i & = &  T_iE_iE_j \quad = \quad E_jT_iE_j\qquad \text{ for  all  $i,j$ such that   $\vert i  -  j\vert =1$}\\
\label{E9}
E_iT_jT_i & = &   T_jT_i E_j \qquad \text{ for  all  $i,j$ such that   $\vert i - j\vert =1$}.
\end{eqnarray}
\end{definition}
\begin{remark}\rm
 The  above  definition  coincides  with  the  original definition of  ${\mathcal E}_n$  under  the  substitution of  $u$  with  $1/u$  and of   $T_i$  with $-T_i$, see \cite{juICTP}.
\end{remark}

\subsection{}\label{explanE}

Behind  the definition of the  algebra of braids and ties ${\mathcal E}_n$ there is the Yokonuma--Hecke algebra ${\rm Y}_{d,n}= {\rm Y}_{d,n}(u) $, where
$d$ denotes a positive integer. We refer to \cite{julaWS2} for the  role  of this algebra in knot theory and to \cite{chpoAM}
for its  combinatorial representation theory. The algebra ${\rm Y}_{d,n}$ can be regarded as a  $u$--deformation of the
wreath product of
 the symmetric group $S_n$ by the cyclic group   $C_d$ of order $d$,  in  an  analogous way   as the Hecke algebra
is a deformation of $S_n$.
More precisely, the algebra ${\rm Y}_{d,n}$ is the algebra generated  by  the braid generators $g_1, \ldots , g_{n-1}$ together with the framing generators $t_1, \ldots , t_{n}$ which   satisfy the following   relations:
the braids relations (said of  type $A$) among the $g_i$'s,  $t_it_j=t_jt_i$,
$g_it_j = t_{s_i(j)}g_i$, $t_i^d=1$  and
\begin{equation}\label{cuadratica}
g_i^2 = 1 + (u-1) e_i (1  + g_i)
\end{equation}
where $e_i$ is defined  by
\begin{equation}\label{eisumti}
e_i :=\frac{1}{d}\sum_{s=1}^{d}t_i^st_{i+1}^{-s}.
\end{equation}
\begin{remark}\rm \label{relationHecke}
Denote by ${\rm H}_n$ the Hecke algebra of parameter $u$, that is, the associative $K$--algebra defined by generators $h_1, \ldots , h_{n-1}$ subject to the braid relations (of type $A$) among the $h_i$'s and the Hecke quadratic relations $h_i^2 = u + (u-1)h_i$, for all $i$. We note that for  $d=1$, the algebra ${\rm Y}_{d,n}$ is the Hecke algebra, since the elements $t_i$  are trivial, so  $e_i=1$  for all $i$, and thus (\ref{cuadratica})  becomes the  quadratic Hecke relation.
 It is now clear that the  mappings  $g_i\mapsto h_i$ and $t_i\mapsto 1$  define an epimorphism from ${\rm Y}_{d,n}$ to ${\rm H}_n$.  We  denote this  epimorphism  by $\phi_n$.

\end{remark}

 The definition of the bt--algebra is obtained by considering abstractly the $K$-- algebra generated by the $g_i$'s and the $e_i$'s. Then   $g_i$ becomes $T_i$, $e_i$ becomes $E_i$ and the set of the  defining relations of the bt--algebra corresponds to  the  complete  set of  relations derived from the commuting relation among the $g_i$'s and the $e_i$'s.   Thus, in particular, we have the following proposition.

\begin{proposition}  \cite[Lemma 2.1]{julaJKTR} \label{EtoYH}
There is a natural   homomorphism  $\psi_n :
{\mathcal E}_n\rightarrow{\rm Y}_{d,n}$ defined through the mappings $T_i\mapsto g_i$ and
$E_i  \mapsto e_i$.
\end{proposition}

\begin{remark}\rm\label{EontoH}

\begin{enumerate}
\item Observe that the composition $\varphi_n:=\phi_n\circ\psi_n$,   sending $T_i\mapsto h_i$ and $E_i\mapsto 1$,  is an epimorphism from ${\mathcal E}_n$   to  ${\rm H}_n$.

\item Assuming as  ground field ${\Bbb C}(q)$, where $q^2= u$, we can deduce from  a theorem of J. Espinoza and S. Ryom--Hansen that   $\psi_n$ is injective  for $d\geq n$, see \cite{esry}. Indeed,   first we note that  in   the   presentations for  ${\rm Y}_{d,n}$ and $\mathcal{E}_n$ used in  \cite{esry}   the   respective braid generators satisfy    quadratic relations modified  with  respect  to the  original in \cite{juJKTR} and \cite{aijuICTP1}.
More precisely,  denoting  by $\tilde{g}_i$ (respectively $\tilde{T}_i) $
the braid generators of the Yokonuma--Hecke algebra (respectively, of the  bt--algebra)  used in \cite{esry},   the quadratic relations  are $\tilde{g}_i^2 = 1 + (q-q^{-1})e_i\tilde{g}_i$
  (respectively
$\tilde{T}_i^2 = 1+(q-q^{-1})E_i\tilde{T}_i$).
Now,  Theorem 8  in \cite{esry} states that  for $d\ge n$  the homomorphism $\psi^{\prime}_n: {\mathcal E}_n\rightarrow{\rm Y}_{d,n}$,   mapping $\tilde{T}_i \mapsto \tilde{g}_i$ and $E_i \mapsto e_i$ is injective.
Then the injectivity of $\psi_n$ comes  from the fact that $\psi_n = I^{-1}\circ\psi_n^{\prime}\circ J$, where $I$ and $J$  are the automorphisms defined as the identity on the non--braid generators and   on the braid generators   by:
$$
I(g_i)= \tilde{g}_i + (q-1)e_i\tilde{g}_i \quad \text{and} \qquad J(T_i) = \tilde{T}_i+ (q-1)E_i\tilde{T}_i
$$
 Observe that
$
I^{-1} (\tilde{g}_i) = g_i + (q^{-1}-1)e_ig_i.
$
 \end{enumerate}
\end{remark}

\subsection{}\label{basisRH}
In the present subsection we   outline   some  useful relations among the defining generators and some algebraic properties of the bt--algebra that  we  will  use in the sequel.

First of all, we  recall that the  fact  that the   $T_i$'s  satisfy  the same  braiding relations as the generators  $s_i$'s of $S_n$ implies, in  virtue of a well--known  result of  Matsumoto, that the following elements $T_w$ are well defined
$$
T_w := T_{i_1}\cdots T_{i_k}
$$
where $w = s_{i_1}\cdots s_{i_k}$ is a reduced expression for $w\in S_n$.

 In  the  following proposition we  list some relations  arising directly from the defining relations of $\mathcal{E}_n$.
We shall use these relations along the paper mentioning  only this proposition.
\begin{proposition}\label{relations}
For all $i,j$, we have:
 \begin{enumerate}
\item[(i)]The elements $T_i$'s are invertible. Moreover,
 \begin{equation}\label{inverse}
T_i^{-1} = T_i + (u^{-1}-1)E_i + (u^{-1}-1)E_i T_i
\end{equation}
\item [(ii)] $T_iT_jT_i^{-1} = T_j^{-1}T_iT_j$, for $\vert i-j\vert= 1$
\item [(iii)]  $ T_i^3  - uT_i^2 - T_i + u = 0$.
\end{enumerate}

\end{proposition}

\smallbreak

Now, we extract  some useful results from 	\cite{rhJAC}. For $i<j$, we define $E_{i,j}$ as
 \begin{equation}\label{Eij}
 E_{i,j} = \left\{\begin{array}{ll}
 E_i & \text{for} \quad j = i +1\\
 T_i \cdots T_{j-2}E_{j-1}T_{j-2}^{-1}\cdots T_{i}^{-1}& \text{otherwise.}
 \end{array}\right.
 \end{equation}

 For any  nonempty  subset $J$     of $\bf n$ we define $E_J=1$  if  $|J|=1$ and
$$
 E_J := \prod_{(i,j )\in J\times J, i<j}  E_{i,j}.
$$

Note that $E_{\{i,j\}} = E_{i,j}$. Also note that
in  \cite[Lemma 4]{rhJAC} it is proved that $E_J$ can be computed as
 \begin{equation} \label{EJ}
E_{J} = \prod_{j\in J,\, j\not= i_0}E_{i_0,j}  \quad  \text{where} \quad i_0={\rm min}(J).
\end{equation}
In a similar  way  one  proves  that
 $E_J$ can be computed, writing $J=\{j_0,j_1,\dots,j_m\}$, with $j_i<j_{i+1}$, as
 \begin{equation}  \label{EJI}
E_{J} = \prod_{i=1}^m E_{j_{i-1}, j_i}.
\end{equation}
Moreover,
for $I = (I_1, \ldots , I_m) \in \mathsf{P}({\bf n})$,  we define
 $E_I$  by
 \begin{equation}\label{EI}
 E_I = \prod_{k}E_{I_k}.
 \end{equation}

The action of $S_n$ on $\mathsf{P}({\bf n})$,  transferred to  the elements $E_{I}$, is  given by the
following formulae

\begin{equation}\label{w(E)}
T_w E_I T_w^{-1} =  E_{w(I)}\qquad (\text{see \cite[Corollary 1]{rhJAC}}).
\end{equation}

\begin{theorem} \cite[Corollary 3]{rhJAC} \label{basEn}
The set $\mathcal{B}_n=\{ T_w E_I \,; \,w\in S_n,\, I\in  \mathsf{P}({\bf n})\}$ is a linear basis of
${\mathcal E}_n$. Hence the dimension of ${\mathcal E}_n$ is $b_nn!$.
\end{theorem}

The following  corollary can be found also in \cite[Proposition
1]{aijuICTP1}, cf. \cite[Section 5]{AJlinks}.

\begin{corollary}\label{atmost}

Any word  in the defining generators of $\mathcal{E}_n$ can be written
as a linear combination of words  in the defining generators of
$\mathcal{E}_n$, containing at most  one element of  the   set $\{ T_{n-1}, E_{n-1}, T_{n-1}E_{n-1} \}$.

\end{corollary}

\begin{remark}\rm
From the natural inclusions  $S_n\subset S_{n+1}$ and   $\mathsf{P}({\bf n})\subset \mathsf{P}({\bf n} \cup \{n+1\})$, it follows that  $\mathcal{B}_n$  can be identified with a linearly independent subset of  $\mathcal{E}_{n+1}$, still denoted   by $\mathcal{B}_n$, which is contained in  $\mathcal{B}_{n+1}$. Then,
 $\mathcal{E}_n$ can be regarded as a subalgebra of
$\mathcal{E}_{n+1}$. In  particular, we derive  the following tower of algebras:
\begin{equation}\label{tower}
\mathcal{E}_1 \subset \mathcal{E}_2 \subset \cdots \subset \mathcal{E}_n \subset \cdots
\end{equation}
Note  that   every element of
$\mathcal{B}_n\mathcal{B}_n$ is a linear combination of  elements of
$\mathcal{B}_n$.
 \end{remark}

\subsection{}\,

  Because  of  (\ref{wcanonico}),
we have that for every $w\in S_n$  the
 element $T_w\in\mathcal{B}_n$  can be written uniquely as
$$
T_w= T_{w_1}T_{w_2}\cdots T_{ w_{n-1}}
$$
where
$$
T_{w_i}\in \{1, T_i, T_iT_{i-1},\ldots , T_iT_{i-1}\cdots T_1 \}.
$$
Set $\mathbb{T}_{i,0}= 1$ and for $k\in\{1, \ldots , i\}$, define
$$
\mathbb{T}_{i,k} = T_iT_{i-1}\cdots T_k.
$$
Thus the elements  of the basis $\mathcal{B}_n$ can be rewritten as
\begin{equation}\label{canonicalform}
\mathbb{T}_{1,k_1}  \mathbb{T}_{2,k_2}\cdots \mathbb{T}_{n-1,k_{n-1}}E_I
\end{equation}
where $ k_j\in\{0,\ldots ,  j\}$ , $j\in \{1, \ldots , n-1\}$ and $I\in  \mathsf{P}({\bf n})$.

\begin{notation} It is convenient to denote  $\check{\mathbb{T}}_{i,k}$  the element obtained by removing $T_i$ from $\mathbb{T}_{i,k}$, that is,
$
\check{\mathbb{T}}_{i,k} = T_{i-1}\cdots T_k.
$
Consequently, $
\check{\check{\mathbb{T}}}_{i,k} = T_{i-2}\cdots T_k.
$
\end{notation}
Using  the   defining relations   of    $\mathcal{E}_n$ we  obtain the following useful relations
\begin{equation}\label{Tt}
\mathbb{T}_{i,k}T_j =
\left\{\begin{array}{ll}
\mathbb{T}_{i,k+1} + (u-1)\mathbb{T}_{i,k+1}E_k( 1+T_k ) & \qquad \text{for}\quad  j=k\\
 \mathbb{T}_{i,j}  & \qquad \text{for}\quad j= k-1\\
 T_j\mathbb{T}_{i,k}  & \qquad \text{for}\quad j\in\{1,\ldots , k-2\}\\
 T_{j-1}\mathbb{T}_{i,k}  & \qquad \text{for}\quad j\in\{k+1,\ldots , i\}.
 \end{array}\right.
\end{equation}
We will    employ  also the following relations  which   are  obtained  using  only  the  braid relations:


\begin{equation}\label{tTT}
T_i\mathbb{T}_{i-1,r}\mathbb{T}_{i,s} =
\left\{\begin{array}{ll}
 \mathbb{T}_{i-1,s-1} \mathbb{T}_{i,r} & \qquad \text{for}\quad 0<r< s\\
\mathbb{T}_{i-1,r} \mathbb{T}_{i,r}  \mathbb{T}_{r,s}  & \qquad \text{for}\quad  r\ge s.\\
  \end{array}\right.
\end{equation}

Notice that:
\begin{eqnarray}\label{tTT1}
\begin{array}{rcl}
\mathbb{T}_{i-1,r} \mathbb{T}_{i,r}  \mathbb{T}_{r,s}   &  = &
\mathbb{T}_{i-1,r} \mathbb{T}_{i,r+1}   \mathbb{T}_{r-1,s} +
(u-1)\mathbb{T}_{i-1,r} \mathbb{T}_{i,r+1} E_r  \mathbb{T}_{r-1,s}\\
& & +
(u-1)\mathbb{T}_{i-1,r} \mathbb{T}_{i,r+1} E_r  \mathbb{T}_{r,s}.
\end{array}
\end{eqnarray}


Also,   from (\ref{w(E)}) we get

\begin{equation}\label{TE_I}
\mathbb{T}_{i,j}E_I = E_{\theta_{i,j}(I)} \mathbb{T}_{i,j}\quad
\end{equation}
where $\theta_{i,j} := s_is_{i-1}\cdots s_{j}$.

For every    $I\in  \mathsf{P}({\bf n})$   and  $k <n$, we  define
\begin{equation}\label{tau} \tau_{n,k}(I)=(I \ast \{k,n\}) \backslash n.\end{equation}

\begin{examples}

Let   $I=(  \{ 1,2,4,6 \}, \{3, 5\} ) \in \mathsf{P}({\bf 6})$,  then
$$\tau_{6,1}(I)= (  \{1,2,4 \},    \{3, 5\} ) \qquad \text{and} \qquad  \tau_{6,3}(I)= ( \{ 1,2 ,3,4, 5\} ).$$

If $I=(  \{ 3, 5,6 \}  )$ then
 $$\tau_{6,3}(I)= (\{ 3,5 \}    )  \qquad \text{and} \qquad  \tau_{6,2}(I)=( \{ 2,3,5 \} ).$$

\end{examples}

\section{Markov trace}\label{sectionmarkovtrace}

In this section we prove that   the bt--algebra  supports a Markov trace. To do this, we use  the method  of relative traces   taking as main reference \cite{chpoArxiv}, (see  also \cite{isog1, isog2, iski}).   Roughly, the method
consists in    defining certain linear maps $\varrho_{n}$, called relative traces, from $\mathcal{E}_{n}$ in $\mathcal{E}_{n-1}$, associated to the tower of the algebras (\ref{tower}).
 Then we prove that the  composition of these  linear maps is indeed the desired Markov trace (see Theorem \ref{trace}).


\subsection{}\label{sectionmarkovtrace1}

From now on we fix two   parameters  $\mathsf{A}$ and $\mathsf{B}$ in $K$.   However, when needed,  we consider $\mathsf{A}$ and $\mathsf{B}$ as variables and  work with the algebras $\mathcal{E}_n\otimes_K K(\mathsf{A}, \mathsf{B})$ which we denote for simplicity by the same symbols $\mathcal{E}_n$.

\begin{definition}\label{rho}
  For every integer $n>1$, let  $\varrho_n$ be  the linear map from ${\mathcal E}_n$ to ${\mathcal E}_{n-1}$  defined on the basis $\mathcal{B}_n$ as follows:
$$
\varrho_n(\mathbb{T}_{1,k_1}  \mathbb{T}_{2,k_2}\cdots \mathbb{T}_{n-1,k_{n-1}}E_I)
= \left\{ \begin{array}{ll}
 \mathbb{T}_{1,k_1}  \mathbb{T}_{2,k_2}\cdots \mathbb{T}_{n-2,k_{n-2}}E_I &  \text{for}\quad k_{n-1}=0,\quad n\not\in \Supp(I)
\\
 \mathsf{B}\mathbb{T}_{1,k_1}  \mathbb{T}_{2,k_2}\cdots \mathbb{T}_{n-2,k_{n-2}}E_{I\backslash n} &   \text{for}\quad k_{n-1}=0, \quad n\in \Supp(I) \\
\mathsf{A}\mathbb{T}_{1,k_1}  \mathbb{T}_{2,k_2}\cdots \check{\mathbb{T}}_{n-1,k_{n-1}}E_{\tau_{n,k_{n-1}}(I)} &
  \text{for}\quad k_{n-1}\not=0.
 \end{array}\right.
$$
 \end{definition}
\smallbreak

Observe that  $\varrho_n$ acts as the identity on $\mathcal{E}_{n-1}$, hence $\varrho_n(1) = 1$, for all $n$.
 Note also that,
from the definition of the $\varrho_n$'s,  it   follows that they satisfy
the following:	
\begin{eqnarray}
\label{RT4}
\varrho_n(T_{n-1}) = \varrho_n(E_{n-1}T_{n-1})
& =  &
\mathsf{A} \\
\label{RT5}
\varrho_n(E_{n-1})
& =  &
\mathsf{B}.
\end{eqnarray}

\begin{remark} \rm
The  reason  of  the  equality  (\ref{RT4}) follows from  the properties of  the  relative  trace  together  with   the  defining relation  (\ref{E8}) of the  bt--algebra. Indeed, assume    that $\A := \varrho_n(E_{n-1}T_{n-1})$ and $\A^{\prime} := \varrho_n( T_{n-1})$;  then
\begin{eqnarray*}
E_{n-2}\varrho_n(E_{n-1}T_{n-1}) &  = & \varrho_n(E_{n-2}E_{n-1}T_{n-1})\\
& = & \varrho_n(E_{n-1}E_{n-2}T_{n-1})\\
& = & \varrho_n(E_{n-2}T_{n-1}E_{n-2})\\
& = & E_{n-2}\varrho_n(T_{n-1})E_{n-2}.
\end{eqnarray*}
Therefore $E_{n-2}\varrho_n(E_{n-1}T_{n-1}) = E_{n-2}\varrho_n(T_{n-1})$, which implies  $\A = \A^{\prime}$.
\end{remark}

 We are going to prove   the following theorem.


\begin{theorem} \label{relativet} The family   $\{ \varrho_n \}_{n>1}$ satisfies,  for all $ X,Z\in {\mathcal E}_{n-1}$ and $Y\in {\mathcal E}_n$:
\begin{eqnarray}
\label{RT1}
\varrho_{n}(XYZ)
& = &
 X\varrho_n(Y)Z \\
\label{RT2}
\varrho_{n-1}(\varrho_n(T_{n-1}Y))
& = &
\varrho_{n-1}(\varrho_n(YT_{n-1}))\\
\label{RT3}
\varrho_{n-1}(\varrho_n(E_{n-1}Y))
& = &
\varrho_{n-1}(\varrho_n(YE_{n-1})).
\end{eqnarray}
\end{theorem}
\begin{proof}
 The Lemma \ref{relativepropertie1} below implies (\ref{RT1}) and
 we will  prove (\ref{RT2}) and (\ref{RT3})  in Lemma \ref{relativepropertie3}.
 \end{proof}


\begin{lemma}\label{relativepropertie1}
For all $ X,Z\in {\mathcal E}_{n-1}$ and $Y\in {\mathcal E}_n$, we have:
\begin{enumerate}
\item[(i)] $\varrho_n( YZ) =  \varrho_n(Y)Z$
\item[(ii)] $\varrho_n( XY) =  X\varrho_n(Y)$.
\end{enumerate}
\end{lemma}
\begin{proof}

From the linearity of $\varrho_n$, it follows that it is enough to prove the  lemma  when  $Y\in \mathcal{B}_n$ and $X, Z$ are the generators $T_1, \ldots ,T_{n-2}$ and $E_1, \ldots ,E_{n-2}$. We set along the proof of the lemma:
$$
Y= \mathbb{T}_{1,k_1}  \mathbb{T}_{2,k_2}\cdots \mathbb{T}_{n-1,k_{n-1}}E_I.
$$
 We prove now  the  claim (i).   We start with  the case in  which $Z$ is one of the generators $T_j$,
with $j\in\{1, \ldots ,n-2\}$.  We  have
\begin{equation}\label{ YZT}
  YZ=  \mathbb{T}_{1,k_1}  \mathbb{T}_{2,k_2}\cdots \mathbb{T}_{n-2,k_{n-2}}\mathbb{T}_{n-1,k_{n-1}}T_jE_{s_j(I)}
\end{equation}
 We shall distinguish now three  cases, labeled below as  Cases I, II and III.

{\bf Case I:}  $k_{n-1}=0$.

In the case $n\not\in \Supp(I)$, the claim follows since $\varrho_n$ acts as the identity. For the case  $n\in \Supp(I)$, we have:
$$
 \varrho_n(Y)Z = \mathsf{B} \mathbb{T}_{1,k_1}  \mathbb{T}_{2,k_2}\cdots \mathbb{T}_{n-2,k_{n-2}}E_{I\backslash n}T_j=\mathsf{B} \mathbb{T}_{1,k_1}
\mathbb{T}_{2,k_2}\cdots \mathbb{T}_{n-2,k_{n-2}}T_jE_{s_j(I\backslash n)}.
$$
On the other hand,  the expression (\ref{ YZT}) of $YZ$  can  be  written as a linear combination of elements of the
form $WE_{s_j(I)}$ with $W\in \mathcal{B}_{n-1}$.
Then, $\varrho_n( YZ) = \mathsf{B}  \mathbb{T}_{1,k_1}  \mathbb{T}_{2,k_2}\cdots \mathbb{T}_{n-2,k_{n-2}}T_jE_{s_j(I)\backslash n}$.
Since $s_j$ does not touch $n$, it follows that  $s_j(I\backslash n) = s_j(I)\backslash n$, hence    $ \varrho_n(Y)Z = \varrho_n( YZ)$.

{\bf Case II:}  $k_{n-1} \not = 0$ and $n\not\in \Supp(I)$.

Now, according to the commutation rules given in (\ref{Tt}), we shall distinguish four subcases.

$\ast$  Subcase $j= k_{n-1}-1$. We  have
\begin{equation}\label{( Y)Z}
 YZ=  \mathbb{T}_{1,k_1}  \mathbb{T}_{2,k_2}\cdots \mathbb{T}_{n-2,k_{n-2}}  \mathbb{T}_{n-1,j}E_{s_j(I)}.
\end{equation}
Since   $n\notin \Supp(s_j(I))$, according  to  Definition \ref{rho}
$$
\varrho_n( YZ)= \mathsf{A} \mathbb{T}_{1,k_1}  \mathbb{T}_{2,k_2}\cdots \mathbb{T}_{n-2,k_{n-2}}\check{\mathbb{T}}_{n-1,j}E_{s_j(I)},
$$
which is equal to $ \varrho_n(Y)Z$. Indeed,
$$
 \varrho_n(Y)Z =  (\mathsf{A}\mathbb{T}_{1,k_1}  \mathbb{T}_{2,k_2}\cdots \check{\mathbb{T}}_{n-1,k_{n-1}}E_{I})T_j =
\mathsf{A} \mathbb{T}_{1,k_1}  \mathbb{T}_{2,k_2}\cdots \check{\mathbb{T}}_{n-1,j}E_{s_j(I)}.
$$

$\ast$  Subcases $j < k_{n-1}-1$ and $k_{n-1}+1  \leq j\leq n-1 $ are totally  analogous to the subcase above.

$\ast$  Subcase $j= k_{n-1}$. We have $ \varrho_n(Y)Z =  \varrho_n(\mathbb{T}_{1,k_1}
\mathbb{T}_{2,k_2}\cdots \mathbb{T}_{n-1,k_{n-1}}E_I)T_j$. Then
\begin{eqnarray*}
 \varrho_n(Y)Z & = & \mathsf{A} \mathbb{T}_{1,k_1}  \mathbb{T}_{2,k_2}\cdots \check{\mathbb{T}}_{n-1,k_{n-1}}T_jE_{s_j(I)}\\
& = & \mathsf{A}  \mathbb{T}_{1,k_1}  \mathbb{T}_{2,k_2}\cdots \check{\mathbb{T}}_{n-1,j+1}T_j^2E_{s_j(I)}
\end{eqnarray*}
By splitting $T_j^2$, we obtain $ \varrho_n(Y)Z  =  W_1 + W_2 +W_3$,  where
\begin{eqnarray*}
W_1 & := & \mathsf{A} \mathbb{T}_{1,k_1}  \mathbb{T}_{2,k_2}\cdots \check{\mathbb{T}}_{n-1,j+1}E_{s_j(I)}\\
W_2 & := & (u-1)\mathsf{A} \mathbb{T}_{1,k_1}  \mathbb{T}_{2,k_2}\cdots \check{\mathbb{T}}_{n-1,j+1}E_jE_{s_j(I)} \\
W_3 & := &  (u-1)\mathsf{A} \mathbb{T}_{1,k_1}  \mathbb{T}_{2,k_2}\cdots \check{\mathbb{T}}_{n-1,j+1}T_jE_jE_{s_j(I)}.
\end{eqnarray*}
On the other hand:
\begin{eqnarray*}
 YZ & = &  \mathbb{T}_{1,k_1}  \mathbb{T}_{2,k_2}\cdots \mathbb{T}_{n-2,k_{n-2}}(\mathbb{T}_{n-1,j+1} + (u-1)\mathbb{T}_{n-1,j+1}E_j( 1+T_j ))E_{s_j(I)}\\
& = &
 W_1^{\prime} + W_2^{\prime} + W_3^{\prime}
\end{eqnarray*}
where
\begin{eqnarray*}
W_1^{\prime} & := &    \mathbb{T}_{1,k_1}  \mathbb{T}_{2,k_2}\cdots \mathbb{T}_{n-2,k_{n-2}}\mathbb{T}_{n-1,j+1} E_{s_j(I)}\\
W_2^{\prime} & := & (u-1)   \mathbb{T}_{1,k_1}  \mathbb{T}_{2,k_2}\cdots \mathbb{T}_{n-2,k_{n-2}}\mathbb{T}_{n-1,j+1}E_jE_{s_j(I)} \\
W_3^{\prime} & := & (u-1)   \mathbb{T}_{1,k_1}  \mathbb{T}_{2,k_2}\cdots \mathbb{T}_{n-2,k_{n-2}}\mathbb{T}_{n-1,j+1}T_jE_jE_{s_j(I)}.
\end{eqnarray*}
Now we observe  that $W_i = \varrho_n(W_i^{\prime})$.  Therefore  $ \varrho_n(Y)Z =\varrho_n( YZ)$.

{\bf Case III:}  $k_{n-1} \not = 0$ and $n\in \Supp(I)$.

Again, we will prove the claim using formulae (\ref{Tt}). Suppose $j= k_{n-1}-1$.  Using   Definition \ref{rho}, we  get
 $$
 \varrho_n( YZ)= \mathsf{A} \mathbb{T}_{1,k_1}  \mathbb{T}_{2,k_2}\cdots \mathbb{T}_{n-2,k_{n-2}}\check{\mathbb{T}}_{n-1,j}E_{\tau_{n,j}(s_j(I))}
 $$
where $\tau_{n,j}(s_j(I)) =( s_j(I)\ast \{j,n\})\backslash n $,  and
 $$ \varrho_n( Y)   =   \mathsf{A}  \mathbb{T}_{1,k_1}  \mathbb{T}_{2,k_2}\cdots \mathbb{T}_{n-2,k_{n-2}}
\check{\mathbb{T}}_{n-1,j+1}E_{ \tau_{n,j+1} (I)  }  $$
where
\begin{equation}\label{Qprime} \tau_{n,j+1} (I)  =( I \ast \{j+1,n\})  \backslash n.
\end{equation}
Observe  that,  since  $j<n-1$,
   $ \tau_{n,j+1} (I)=s_{n-1}(\tau_{n,j}(s_j(I)))  $. Therefore we  have
\begin{eqnarray*}
\varrho_n( YZ) & = & \mathsf{A} (\mathbb{T}_{1,k_1}  \mathbb{T}_{2,k_2}\cdots \mathbb{T}_{n-2,k_{n-2}}
\check{\mathbb{T}}_{n-1,j}E_{\tau_{n,j}(s_j(I))})\\
 & = & \mathsf{A} (\mathbb{T}_{1,k_1}  \mathbb{T}_{2,k_2}\cdots \mathbb{T}_{n-2,k_{n-2}}
\check{\mathbb{T}}_{n-1,j+1}T_jE_{\tau_{n,j}(s_j(I))})\\
& = & \mathsf{A} (\mathbb{T}_{1,k_1}  \mathbb{T}_{2,k_2}\cdots \mathbb{T}_{n-2,k_{n-2}}
\check{\mathbb{T}}_{n-1,j+1}E_{\tau_{n,j+1}  (I) })T_j\\
& = &
 \varrho_n(Y)Z.
\end{eqnarray*}
The cases  $j < k_{n-1}-1$ and $k_{n-1}+1  \leq j\leq n-1 $ are verified in analogous way.

Suppose now $j= k_{n-1}$. We have
$$ YZ=  \mathbb{T}_{1,k_1}  \mathbb{T}_{2,k_2}\cdots \mathbb{T}_{n-1,j}T_jE_{s_j(I)}$$
and
$$
  \varrho_n(YZ)= \varrho_n( \mathbb{T}_{1,k_1}  \mathbb{T}_{2,k_2}\cdots \mathbb{T}_{n-1,j+1}T_j^2 E_{s_{j}(I)})
  =
 V_1 + V_2 + V_3
$$
being
\begin{eqnarray*}
V_1 & := & \mathsf{A} \mathbb{T}_{1,k_1}  \mathbb{T}_{2,k_2}\cdots \check{\mathbb{T}}_{n-1,j+1}E_{\tau_{n,j+1}(s_j (I))}\\
V_2 & := & (u-1)\mathsf{A} \mathbb{T}_{1,k_1}  \mathbb{T}_{2,k_2}\cdots \check{\mathbb{T}}_{n-1,j+1}E_{\tau_{n,j+1}(s_j (I)) }E_j \\
V_3 & := &  (u-1)\mathsf{A} \mathbb{T}_{1,k_1}  \mathbb{T}_{2,k_2}\cdots \check{\mathbb{T}}_{n-1,j+1}T_j E_{ \tau_{n,j }(s_j(I))}E_j.
\end{eqnarray*}

 On  the  other  hand,  we  have
 \begin{eqnarray*}
  \varrho_n(Y)Z
 &=&
 \mathsf{A} (\mathbb{T}_{1,k_1}  \mathbb{T}_{2,k_2}\cdots \check{\mathbb{T}}_{n-1,j} E_{\tau_{n,j}(I)})T_j\\
 & = &
 \mathsf{A} \mathbb{T}_{1,k_1}  \mathbb{T}_{2,k_2}\cdots \check{\mathbb{T}}_{n-1,j+1}T_j^2E_{s_j(\tau_{n,j} (I ))}.\\
\end{eqnarray*}

  Splitting  $T_j^2$,  we  obtain $\varrho_n(Y)Z=V'_1+V'_2+V'_3$, where
\begin{eqnarray*}
V'_1 & = & \mathsf{A} \mathbb{T}_{1,k_1}  \mathbb{T}_{2,k_2}\cdots \check{\mathbb{T}}_{n-1,j+1}  E_{s_j(\tau_{n,j} (I ))}\\
V'_2 & = & (u-1)\mathsf{A} \mathbb{T}_{1,k_1}  \mathbb{T}_{2,k_2}\cdots \check{\mathbb{T}}_{n-1,j+1}E_jE_{s_j(\tau_{n,j} (I ))} \\
V'_3 & = &  (u-1)\mathsf{A} \mathbb{T}_{1,k_1}  \mathbb{T}_{2,k_2}\cdots \check{\mathbb{T}}_{n-1,j+1}T_jE_jE_{s_j(\tau_{n,j} (I ))}.
\end{eqnarray*}

We  have  therefore  to  verify   that  $V_i=V'_i$,  $i=1,2,3$.
$V'_1=V_1$  and  $V'_2=V_2$  since
 $$   s_j(\tau_{n,j} (I ))  =\tau_{n,j+1}(s_j (I)) .$$
 As for $V'_3$,  we  have
 $$ E_jE_{s_j(\tau_{n,j} (I ))}=E_{s_j(\tau_{n,j} (I ))  \ast \{j,j+1\}}, $$
 and
$$ s_j(\tau_{n,j} (I ))  \ast \{j,j+1\} =   s_j ( (I\ast \{j, n\})  \backslash n  )\ast \{j,j+1\}.$$
This  partition  is  the  same  as  that   in  the  expression of $V_3$, namely
 $$ \tau_{n,j }(s_j(I))\ast\{j,j+1\}   = ((  s_j(I)\ast \{j,n\})    \backslash n )\ast \{j,j+1\},  $$
since  $j<n-1$.
 Thus  we have  also $V_3=V'_3$.

To finish the proof of (i)   it  is left only to consider the case when   $Z=E_j$. We  have
 \begin{equation}\label{ YZE}
  YE_j=  \mathbb{T}_{1,k_1}  \mathbb{T}_{2,k_2}\cdots \mathbb{T}_{n-2,k_{n-2}}\mathbb{T}_{n-1,k_{n-1}}E_{I \ast j }.
\end{equation}
Observe  that  $ (I\backslash n )  \ast j = ( I \ast j )\backslash n$,  because $j<n-1$.
Applying      Definition \ref{rho}, we  get in  all  cases  $\varrho_n(Y)Z=\varrho_n(YZ)$  since at  the  end  of  the  left  and  right  sides we  have  respectively     $E_{(I \ast j) \backslash n} $
and  $E_{(I\backslash n )\ast j}$.

\smallbreak

 Now  we prove the claim (ii) of  the  lemma.  In the case $k_{n-1}=0$ and $n\not\in \Supp(I)$ the claim is  evident, since $Y\in {{\mathcal E}_{n-1}}$  and $\varrho_n$ acts as the identity on ${{\mathcal E}_{n-1}}$.

 In the case $k_{n-1}=0$ and $n\in \Supp(I)$, we have $Y = \mathbb{T}_{1,k_1}  \mathbb{T}_{2,k_2}\cdots \mathbb{T}_{n-2,k_{n-2}}E_I$.  Then
 $$
X\varrho_n(Y) = X  \mathbb{T}_{1,k_1}  \mathbb{T}_{2,k_2}\cdots \mathbb{T}_{n-2,k_{n-2}}E_{I\backslash n}.
 $$
 Now, to compute $\varrho_n(XY)$, we need
 to express $XY$ as linear combination of   elements of the basis $\mathcal{B}_n$,
 but in the case we are considering  it is enough to express $X^{\prime}:=X\mathbb{T}_{1,k_1}  \mathbb{T}_{2,k_2}\cdots \mathbb{T}_{n-2,k_{n-2}}$ as linear combination of   elements of $\mathcal{B}_{n-1}$, and then to put the element $E_I$ on the right of each term  of this linear combination. Thus,  $\varrho_n(XY)$ is the linear combination  obtained from  the  linear combination expressing  $X^{\prime}$,  by putting   on the right of each term the factor $E_{I\backslash n}$. Hence, we  deduce  that $X\varrho_n(Y) = \varrho_n(XY)$.

Suppose now  that $k_{n-1}\not=0$. Firstly, we   check   the claim for  $X=T_m$, where  $m\in\{1, \ldots ,n-2\}$.

We have $X\varrho_n(Y) = \mathsf{A}  T_m \mathbb{T}_{1,k_1}   \cdots \check{\mathbb{T}}_{n-1,k_{n-1}}E_{\tau_{n,k_{n-1}}(I)}$.
We  rewrite  it  as
\begin{equation}\label{trhoY}
 X \varrho_n(Y)= \mathsf{A}\mathbb{A} (T_m\mathbb{T}_{m-1,r}\mathbb{T}_{m,s}) \mathbb{B}
\end{equation}
where
\begin{eqnarray*}
\mathbb{A}
& := &
\mathbb{T}_{1,k_1}  \mathbb{T}_{2,k_2}\cdots \mathbb{T}_{m-2,k_{m-2}}\\
\mathbb{B}
& := &
\mathbb{T}_{m+1,k_{m+1}}\cdots \mathbb{T}_{n-2,k_{n-2}}\check{\mathbb{T}}_{n-1,k_{n-1}}E_{\tau_{n,k_{n-1}}(I)}
\end{eqnarray*}

On the other hand, we have
\begin{equation}\label{tY}
 X Y= \mathbb{A} (T_m\mathbb{T}_{m-1,r}\mathbb{T}_{m,s}) \mathbb{B}^{\prime}
\end{equation}
where   $0\leq r \leq m-1$, $0< s\leq m$ and
$$
\mathbb{B}^{\prime} := \mathbb{T}_{m+1,k_{m+1}}\cdots \mathbb{T}_{n-2,k_{n-2}}\mathbb{T}_{n-1,k_{n-1}}E_I.
$$

We will compare now $\varrho_n(XY)$ with $X\varrho_n(Y)$, distinguishing the cases  $r=0$ and $r\not=0$.


{\bf Case $r\not=0$}. By using   (\ref{tTT}) and later (\ref{tTT1}) we deduce:
$$
X\varrho_n(Y) =
\left\{\begin{array}{ll}
\mathsf{A}\mathbb{A} R \mathbb{B}& \qquad \text{for}\quad 0<r\leq s\\
\mathsf{A}\mathbb{A}S_1 \mathbb{B}+ (u-1)\mathsf{A}\mathbb{A}S_2 \mathbb{B}+ (u-1)\mathsf{A}\mathbb{A}S_3 \mathbb{B}& \qquad \text{for}\quad  s<r
\end{array}\right.
$$
where
\begin{eqnarray*}
R
& := &
\mathbb{T}_{m-1,s-1} \mathbb{T}_{m,r}\\
S_1
& := &
\mathbb{T}_{m-1,r} \mathbb{T}_{m,r+1} \mathbb{T}_{r-1,s}=\mathbb{T}_{m-1,s} \mathbb{T}_{m,r+1}  \\
S_2
& := &
\mathbb{T}_{m-1,r}  \mathbb{T}_{m,r+1} E_r\mathbb{T}_{r-1,s}=\mathbb{T}_{m-1,s}  \mathbb{T}_{m,r+1} E_{\{a,b\}}\\
S_3
& := &
\mathbb{T}_{m-1,r}  \mathbb{T}_{m,r+1} T_rE_r \mathbb{T}_{r-1,s}=\mathbb{T}_{m-1,r}  \mathbb{T}_{m,s}  E_{\{a,b\}}
\end{eqnarray*}
being $\{a,b\}= \theta^{-1}_{r-1,s}(\{r,r+1 \})=\{s,r+1\}$.
Now, by using again (\ref{tTT}) and later (\ref{tTT1}), we get:
$$
XY =
\left\{\begin{array}{ll}
\mathbb{A} R \mathbb{B}^{\prime}&\qquad \text{for} \quad 0<r\leq s\\
\mathbb{A}S_1 \mathbb{B}^{\prime}+ (u-1) \mathbb{A}S_2 \mathbb{B}^{\prime}+ (u-1)\mathbb{A}S_3 \mathbb{B}^{\prime}& \qquad \text{for}\quad  s<r
\end{array}\right.
$$
Then
$$
\varrho_n(XY) =
\left\{\begin{array}{ll}
  \A  \mathbb{A} R \mathbb{B} & \qquad \text{for} \quad 0<r\leq s  \\
\varrho_n(\mathbb{A}S_1 \mathbb{B}^{\prime})+ (u-1)\varrho_n(\mathbb{A}S_2 \mathbb{B}^{\prime})+ (u-1)\varrho_n(\mathbb{A}S_3 \mathbb{B}^{\prime})& \qquad \text{for}\quad  s<r
\end{array}\right.
$$
Clearly  $\A \mathbb{A} S_1\mathbb{B} = \varrho_n(\mathbb{A}S_1 \mathbb{B}^{\prime})$.   Now,
using (\ref{TE_I}), we obtain
\begin{eqnarray*}
\mathbb{A}S_2 \mathbb{B}^{\prime}
& = &
 \mathbb{A} (\mathbb{T}_{m-1,s} \mathbb{T}_{m,r+1} )  \mathbb{T}_{m+1,k_{m+1}}\cdots \mathbb{T}_{n-2,k_{n-2}}\mathbb{T}_{n-1,k_{n-1}}E_{\{a',b'\}}E_I
 \end{eqnarray*}
 where $\{a',b'\} := \theta_{n-1,k_{n-1}}^{-1}\cdots\theta_{m+1, k_{m+1}}^{-1}  (\{a,b\})$.

  Now,  we  have,
$$
\varrho_n(\mathbb{A}S_2 \mathbb{B}^{\prime})
 =
\mathsf{A}\mathbb{A} ( \mathbb{T}_{m-1,s} \mathbb{T}_{m,r+1} )  \mathbb{T}_{m+1,k_{m+1}}\cdots \mathbb{T}_{n-2,k_{n-2}}\check{\mathbb{T}}_{n-1,k_{n-1}} E_{\tau_{n,k_{n-1}}(I\ast\{a'b'\})},
$$
that is  equal  to  $\mathsf{A}\mathbb{A}S_2 \mathbb{B}$  if
$$ E_{\{a,b\}}\mathbb{B}=\mathbb{T}_{m+1,k_{m+1}}\cdots \mathbb{T}_{n-2,k_{n-2}}\check{\mathbb{T}}_{n-1,k_{n-1}} E_{\tau_{n,k_{n-1}}(I\ast\{a'b'\})}.$$
But
$$E_{\{a,b\}}\mathbb{B}=\mathbb{B}E_{\{a'',b''\} }= \mathbb{T}_{m+1,k_{m+1}}\cdots \mathbb{T}_{n-2,k_{n-2}}\check{\mathbb{T}}_{n-1,k_{n-1}}E_{(\tau_{n,k_{n-1}}(I))\ast \{a'',b''\}} $$
where $\{a'',b''\}=\theta_{n-2,k_{n-1}}^{-1}\cdots\theta_{m+1, k_{m+1}}^{-1} (\{a,b\})$.

 Therefore, we  have  to  check  that
\begin{equation}\label{abab}(\tau_{n,k_{n-1}}(I))\ast \{a'',b''\} =  \tau_{n,k_{n-1}}(I \ast \{a',b'\}). \end{equation}

 Remember  that $r< n-2$  and $s<r$,  so  $b<(n-1)$.  Thus, the  pair $\{a',b'\}$  may contain  $n$,  while the  pair  $\{a'',b''\}$  cannot.

 Observe also  that
 $$\{a'',b''\}=\theta_{n-2,k_{n-1}}^{-1} \theta_{n-1,k_{n-1}} \{a',b'\}  $$
 and $\theta_{n-2,k_{n-1}}^{-1} \theta_{n-1,k_{n-1}}  $  is  the  transposition $(n,k_{n-1})$.

Thus  we  have  to check  (\ref{abab}) in  two cases:

a) if  $\{a',b'\}$  does not  contain neither $k_{n-1}$ nor $n$,  then  $\{a'',b''\}=\{a',b'\}$  and  (\ref{abab})  reads
 $$((I\ast \{n,k_{n-1}\})\backslash n)\ast \{a',b'\}= ((I \ast \{a',b'\})\ast \{n,k_{n-1}\})\backslash n  $$
 which  is  evidently  satisfied.

b) If  $\{a',b'\}=\{a',n\}$,   then  $\{a'',b''\}=\{a',k_{n-1}\}$  and  (\ref{abab})  reads
 $$((I\ast \{n,k_{n-1}\})\backslash n)\ast \{a',k_{n-1}\}= ((I \ast \{a',n\})\ast \{n,k_{n-1}\})\backslash n.  $$
Since $a'<n$, both  terms  are  equal  to
$(I\ast \{a',n,k_{n-1}\})\backslash n$ and (\ref{abab}) is  satisfied.

In  a similar way  we check that $\varrho_n(\mathbb{A}S_3 \mathbb{B}^{\prime}) =
\mathsf{A}\mathbb{A}S_3 \mathbb{B}$. Therefore, $X\varrho_n(Y) = \varrho_n(XY)$ whenever
$r\not =0$.

\smallbreak

{\bf Case $r=0$}. We have that (\ref{trhoY}) becomes $\mathsf{A}\mathbb{A}T_m\mathbb{T}_{m,s}\mathbb{B} = \mathsf{A}\mathbb{A}T_m^2\mathbb{T}_{m-1,s}\mathbb{B}$. So,
$$
T_m\varrho_n(Y) = \mathsf{A}\mathbb{A}\mathbb{T}_{m-1,s}\mathbb{B} +
(u-1)\mathsf{A}\mathbb{A}E_m\mathbb{T}_{m-1,s}\mathbb{B} +
(u-1)\mathsf{A}\mathbb{A}E_m\mathbb{T}_{m,s}\mathbb{B}.
$$
Now, (\ref{tY}) becomes $\mathbb{A}(T_m \mathbb{T}_{m,s})	\mathbb{B}^{\prime} = \mathbb{A}(T_m^2 \mathbb{T}_{m-1,s})	\mathbb{B}^{\prime}$. Then
$$
T_m Y =
\mathbb{A}\mathbb{T}_{m-1,s}	\mathbb{B}^{\prime}
+ (u-1)\mathbb{A}(E_m\mathbb{T}_{m-1,s})	\mathbb{B}^{\prime}
+ (u-1)\mathbb{A}(E_m \mathbb{T}_{m,s})	\mathbb{B}^{\prime}.
$$
The  equality $\varrho_n(T_m Y)= T_m\varrho_n(Y) $ is  thus  obtained as  in  the  previous  case  comparing  the  three  terms in both  members of  the equality.


 \smallbreak
Finally we   check the case (ii) when $X= E_m$,  with $1\le m \le n-2$.  Let
$$
Y= \mathbb{T}_{1,k_1}  \mathbb{T}_{2,k_2}\cdots \mathbb{T}_{n-1,k_{n-1}}E_I.
$$

First  case: $k_{n-1} =0$.  Since  $Y\in \E_n$,  $I\backslash n \not=I $.
 We  have    $$ XY  = \mathbb{T}_{1,k_1}  \mathbb{T}_{2,k_2}\cdots \mathbb{T}_{n-2,k_{n-2}}E_{I}E_{ \{a,b\}} $$
 where  $\{a,b\}=  \theta_{n-2,k_{n-2}}^{-1} \cdots \theta_{2,k_2}^{-1} \theta_{1,k_1}^{-1}(\{m,m+1\})$.
 Moreover $E_{I}E_{ \{a,b\}}=E_{I \ast \{a,b\}} $.

Therefore,    $$\varrho_n( XY)= \B\mathbb{T}_{1,k_1}  \mathbb{T}_{2,k_2}\cdots  \mathbb{T}_{n-2,k_{n-2}}E_{(I\ast \{a,b\} )\backslash n}. $$
On  the  other  hand,  we  have
\begin{eqnarray*} X\varrho_n(Y) &= &E_m \B \mathbb{T}_{1,k_1}  \mathbb{T}_{2,k_2}\cdots \check{\mathbb{T}}_{n-2,k_{n-2}}E_{I\backslash n} \\
&= & \B \mathbb{T}_{1,k_1}  \mathbb{T}_{2,k_2}\cdots \check{\mathbb{T}}_{n-2,k_{n-1}}E_{(I\backslash n)\ast \{a,b\} }.
\end{eqnarray*}
Since  $m\le n-2$,  $a$ and  $b$ cannot  be  higher  than $n-1$,  therefore $(I\ast \{a,b\})\backslash n =(I\backslash n)\ast \{a,b\}$,  so that we  get $\varrho_n(XY)= X\varrho_n(Y)$.

 Second  case:  $k_{n-1}\not=0$.
 We  have    $$ XY  = \mathbb{T}_{1,k_1}  \mathbb{T}_{2,k_2}\cdots \mathbb{T}_{n-1,k_{n-1}}E_{I\ast \{a,b\}}$$
 where  $\{a,b\}=  \theta_{n-1,k_{n-1}}^{-1} \cdots\theta_{2,k_2}^{-1} \theta_{1,k_1}^{-1}(\{m,m+1\})$
and $ E_{I\ast \{a,b\}}=E_I E_{\{a,b\}} $.

Therefore,    $$\varrho_n( XY)= \A \mathbb{T}_{1,k_1}  \mathbb{T}_{2,k_2}\cdots \check{\mathbb{T}}_{n-1,k_{n-1}}E_{\tau_{n,k_{n-1}}(I\ast \{a,b\} )}.$$

 On  the  other  hand,  we  have
\begin{eqnarray*} X\varrho_n(Y) &= &E_m \A \mathbb{T}_{1,k_1}  \mathbb{T}_{2,k_2}\cdots \check{\mathbb{T}}_{n-1,k_{n-1}}E_{\tau_{n,k_{n-1}}(I)} \\
&= & \A \mathbb{T}_{1,k_1}  \mathbb{T}_{2,k_2}\cdots \check{\mathbb{T}}_{n-1,k_{n-1}}E_{\tau_{n,k_{n-1}}(I)\ast \{c,d\}}
\end{eqnarray*}
where  $\{c,d\}= \theta_{n-2,k_{n-1}}^{-1} \cdots \theta_{2,k_2} \theta_{1,k_1} (\{m,m+1\})$. Now,    $\varrho_n(XY)=X\varrho_n(Y)$    if the  two  partitions
 $ \tau_{n,k_{n-1}}(I\ast \{a,b\})$  and  $\tau_{n,k_{n-1}}(I)\ast \{c,d\}$  are  equal, i.e.,  if
\begin{equation}\label{par2}( (I\ast\{a,b\} )\ast \{ k_{n-1},n\} ) \backslash n= ( (I\ast \{ k_{n-1},n\} ) \backslash n)   \ast \{c,d\}.\end{equation}
 Observe that
$$
\{c,d\}=\theta_{n-2,k_{n-1}}^{-1}(\theta_{n-1,k_{n-1}}\{a,b\}),
$$
i.e., $\{c,d\}$ is  obtained  applying to $\{a,b\}$  the  transposition $(k_{n-1},n)$.
Observe  that,  since $m\le n-2$,   $\{a,b\}$ may  contain $n$,  while $\{c,d\}$ cannot.  Therefore   equation (\ref{par2}) is  proved  exactly  as equation (\ref{abab}). Thus,  the  proof of (ii)  is  finished.

\smallbreak

\end{proof}


\begin{lemma}\label{relativepropertie3}
For all $X\in {\mathcal E}_{n}$, we have:

\begin{enumerate}
\item [(i)]$\varrho_{n-1}(\varrho_n(E_{n-1}X)) =
\varrho_{n-1}(\varrho_n(XE_{n-1}))$
 \item [(ii)]$\varrho_{n-1}(\varrho_n(T_{n-1}X)) =
\varrho_{n-1}(\varrho_n(XT_{n-1}))$
\item [(iii)]$\varrho_{n-1}(\varrho_n(T_{n-1}E_{n-1}X)) =
\varrho_{n-1}(\varrho_n(XT_{n-1}E_{n-1}))$.

 \end{enumerate}
 \end{lemma}

\smallbreak

\begin{proof}
 Without loss of  generality, we can suppose $X\in \mathcal{B}_n$. Set
 $$
X=\mathbb{T}_{1,k_1}  \mathbb{T}_{2,k_2}\cdots\mathbb{T}_{n-2,k_{n-2}}\mathbb{T}_{n-1,k_{n-1}}E_J.
 $$

We  prove first the claim (i). Invoking   Lemma \ref{relativepropertie1}, we get
\begin{eqnarray*}
\varrho_{n-1}(\varrho_n(E_{n-1}X))
& = &
\mathbb{T}_{1,k_1}  \mathbb{T}_{2,k_2}\cdots \mathbb{T}_{n-3,k_{n-3}}\varrho_{n-1}(\varrho_n(E_{n-1}\mathbb{T}_{n-2,k_{n-2}}\mathbb{T}_{n-1,k_{n-1}}E_J))
 \\
\varrho_{n-1}(\varrho_n(XE_{n-1}))
& = &
\mathbb{T}_{1,k_1}  \mathbb{T}_{2,k_2}\cdots \mathbb{T}_{n-3,k_{n-3}}\varrho_{n-1}(\varrho_n(\mathbb{T}_{n-2,k_{n-2}}\mathbb{T}_{n-1,k_{n-1}}E_{n-1}E_J))
\end{eqnarray*}
Thus, it is enough to prove that $E=F$, where
\begin{eqnarray*}
E & := & \varrho_{n-1}(\varrho_n(E_{n-1}\mathbb{T}_{n-2,k_{n-2}}\mathbb{T}_{n-1,k_{n-1}}E_J))\\
F & :=& \varrho_{n-1}(\varrho_n(\mathbb{T}_{n-2,k_{n-2}}\mathbb{T}_{n-1,k_{n-1}}E_{n-1}E_J)).
\end{eqnarray*}
To do that, we consider four cases,  distinguishing if  $k_{n-1}$ and  $k_{n-2}$ vanish or not.  In the case $k_{n-1}=k_{n-2}=0$ it is  evident that $E=F$.

\noindent{\bf Case $k_{n-1}=0$ and $k_{n-2}\not=0$.} We have
$$
F= \varrho_{n-1}(\varrho_n(\mathbb{T}_{n-2,k_{n-2}}E_{n-1}E_J))=\mathsf{B}
\varrho_{n-1}(\mathbb{T}_{n-2,k_{n-2}}E_{(J\ast (n-1))\backslash n})).
$$

On the other part
\begin{eqnarray*}
E=\varrho_{n-1} (\varrho_n(E_{n-1}\mathbb{T}_{n-2,k_{n-2}}E_J))
& = &
\varrho_{n-1} (\varrho_n(\mathbb{T}_{n-2,k_{n-2}}E_{\theta_{n-2, k_{n-2}}^{-1}(\{n-1,n\})}E_J))\\
& = &
\varrho_{n-1} (\mathbb{T}_{n-2,k_{n-2}}\varrho_n(E_{\theta_{n-2, k_{n-2}}^{-1}(\{n-1,n\})}E_J)).
\end{eqnarray*}
Now, we have  $\theta_{n-2, k_{n-2}}^{-1}(\{n-1,n\})= \{k_{n-2} , n\}$. So,  we  get
$$\varrho_n(E_{\{k_{n-2} , n\} }E_J) = \mathsf{B} E_{(J\ast \{k_{n-2} , n\}) \backslash n}.$$

In  the  case  in  which $n\notin \Supp(J)$,  evidently:
$$(J\ast \{n-1 , n\}) \backslash n=  (J\ast \{k_{n-2} , n\}) \backslash n=J  $$
so  that $E=F$.
In  the  case    in  which  $J$   contains  a  set $\{a, \dots, n\}$, i.e. $J=(\check J, \{a, \dots, n\} )$,
$$(J\ast \{n-1 , n\}) \backslash n= \{\check J  \ast \{a, \dots, n-1 \} \}=:J_1   $$
$$  (J\ast \{k_{n-2} , n\}) \backslash n=\{ \check J \ast  \{a, \dots, k_{n-2}\} \} =: J_2.$$

Now:
$$  F = B(\varrho_{n-1}(\mathbb{T}_{n-2,k_{n-2}}E_{J_1}) $$ and
$$  E= B(\varrho_{n-1}(\mathbb{T}_{n-2,k_{n-2}}E_{J_2}  ).$$
We  have, for $i=1,2$. $$\varrho_{n-1}(\mathbb{T}_{n-2,k_{n-2}}E_{J_i})=A \check{ \mathbb{T}}_{n-2,k_{n-2}} E_{ J'_i } $$
where $ J'_i =  \tau_{n-1,k_{n-2}}(J_i)= (J_i \ast \{n-1,k_{n-2}\}) \backslash (n-1).$
Evidently $J'_1=J'_2$, since  $J_1$  and  $J_2$ are  identical  up  to the  transposition of  $(n-1)$ with $k_{n-1}$. Therefore  $F=E$.

\smallbreak

\noindent{\bf Case $k_{n-1}\not=0$ and $k_{n-2}=0$.} We have
$$
F= \varrho_{n-1}(\varrho_n(\mathbb{T}_{n-1,k_{n-1}}E_{n-1}E_J))
\quad \text{and} \quad
E=\varrho_{n-1} (\varrho_n(E_{n-1}\mathbb{T}_{n-1,k_{n-1}}E_J)).
$$
But, $E_{n-1}\mathbb{T}_{n-1,k_{n-1}}E_J= \mathbb{T}_{n-1,k_{n-1}}E_{\{k_{n-1}, n\}}E_J $, since  $\theta^{-1}_{n-1, k_{n-1}}(\{n-1,n\})=\{k_{n-1},n\}$. Then
$$
 E=\varrho_{n-1} (\varrho_n(\mathbb{T}_{n-1,k_{n-1}}E_{J \ast \{k_{n-1}, n\}})) =\varrho_{n-1} (\A \, \check{\mathbb{T}}_{n-1,k_{n-1}}E_{J_1})$$
 and
 $$ F=\varrho_{n-1} (\varrho_n(\mathbb{T}_{n-1,k_{n-1}}E_{J \ast \{(n-1), n\}})) =\varrho_{n-1} (\A \, \check{\mathbb{T}}_{n-1,k_{n-1}}E_{J_2}) $$
where $J_1= ( J \ast \{k_{n-1}, n\} )\backslash n $  and $J_2=(  (J \ast \{(n-1), n\})\ast \{k_{n-1}, n\})\backslash n  $.

Thus  $E$ and $F$  can be  written as  follows (for $i=1$ and  $i=2$ respectively)
$$   \A \, \varrho_{n-1} ( \check{\mathbb{T}}_{n-1,k_{n-1}}E_{J_i})=\A^2 \, \check {\check{\mathbb{T}}}_{n-1,k_{n-1}}E_{J'_i} $$
 where  $ J'_i =   (J_i \ast \{k_{n-1},(n-1)\})\backslash (n-1)$.

 Hence the equality $E=F$  follows,  as  in  the  preceding  case,  from  the  fact  that
$J'_1=J'_2 $.

\smallbreak

\noindent{\bf Case $k_{n-1}\not=0$ and $k_{n-2}\not=0$.}
From Lemma \ref{relativepropertie1}, we get
$F = \varrho_{n-1}(\mathbb{T}_{n-2,k_{n-2}}\varrho_n(\mathbb{T}_{n-1,k_{n-1}}E_{n-1}E_J))$. Then
$$
F= \mathsf{A}\, \varrho_{n-1}(\mathbb{T}_{n-2,k_{n-2}}\check{\mathbb{T}}_{n-1,k_{n-1}}
E_{J_1})
$$
where
$J_1=\tau_{n,k_{n-1}}(J\ast (n-1))=  ( J\ast \{k_{n-1},(n-1),n\}) )\backslash n $.

On the other side, $E_{n-1}\mathbb{T}_{n-2,k_{n-2}}\mathbb{T}_{n-1,k_{n-1}}E_J=
\mathbb{T}_{n-2,k_{n-2}}E_{\{k_{n-2}, n\}}\mathbb{T}_{n-1,k_{n-1}} E_J$.

Call $\{a,b\}=\theta^{-1}_{n-1,k_{n-1}}(\{k_{n-2},n\})$.  Observe  that $\{a,b\}=\{k_{n-2},k_{n-1}\}$  if  $k_{n-2}<k_{n-1}$,  whereas  $\{a,b\}=\{k_{n-1},k_{n-2}+1\}$  if  $k_{n-2}\ge k_{n-1}$.

Using  Lemma \ref{relativepropertie1}, we obtain
 \begin{eqnarray*}
 E & = & \varrho_{n-1}(\mathbb{T}_{n-2,k_{n-2}}\varrho_n(E_{\{k_{n-2}, n\}}\mathbb{T}_{n-1,k_{n-1}} E_J)) \\
  & = &    \varrho_{n-1}(\mathbb{T}_{n-2,k_{n-2}}  \varrho_n({\mathbb{T}}_{n-1,k_{n-1}} E_{\{  a,b \}} E_J))  \\
  & = &  \A \, \varrho_{n-1}(\mathbb{T}_{n-2,k_{n-2}}  \check{\mathbb{T}}_{n-1,k_{n-1}} E_{J_2})
 \end{eqnarray*}
being
$J_2=  \tau_{n,k_{n-1}} (J\ast \{a,b\})= (J \ast \{a,  k_{n-1}, n \})\backslash n$, where
 $a=k_{n-2}$  if $k_{n-2}< k_{n-1}$ and  $a=k_{n-2}+1$ otherwise.

 Now,  $J_1\not=J_2$,  so   we  have  to compare $ \varrho_{n-1}(\mathbb{T}_{n-2,k_{n-2}} \mathbb{T}_{n-2,k_{n-1}} E_{J_i}) $, for $i=1,2$. To  calculate   $\varrho_{n-1}$, it is  convenient to  write $\mathbb{T}_{n-2,k_{n-2}}\mathbb{T}_{n-2,k_{n-1}}$  as
  $$\mathbb{T}_{n-2,k_{n-2}}\mathbb{T}_{n-2,k_{n-1}}=T_{n-2}T_{n-3}T_{n-2} \check{\check{\mathbb{T}}}_{n-2,k_{n-2}} \check{\mathbb{T}}_{n-2,k_{n-1}}.$$
Then, using the relation $ T_{n-2}T_{n-3}T_{n-2}=T_{n-3}T_{n-2}T_{n-3}$, and  Lemma \ref{relativepropertie1}, we  get
  $$ \varrho_{n-1}(\mathbb{T}_{n-2,k_{n-2}} \mathbb{T}_{n-2,k_{n-1}} E_{J_i})=
 T_{n-3}\ \varrho_{n-1}(T_{n-2}E_{J'_i}) \  T_{n-3} \check{\check{\mathbb{T}}}_{n-2,k_{n-2}} \check{\mathbb{T}}_{n-2,k_{n-1}}   $$
where  $J'_i= \Theta (J'_i)$,  being  $\Theta=\theta_{n-3,k_{n-2}}\theta_{n-3,k_{n-1}}$.
Let $\{m,\dots,n\}$ be the  set  of  the  partition $J$  containing  $n$, and  denote $\check J$ the  set--partition  obtained  from  $J$ by removing the set $\{m,\dots,n\}$. Then,
$$J_1= (J\ast \{ k_{n-1},n-1, n\} )\backslash n= \check{J}\ast \{m,\dots,k_{n-1},n-1  \}  $$
$$ J_2= (J \ast \{a,k_{n-1}, n \})\backslash n=\check{J}\ast \{m,\dots,a, k_{n-1} \}. $$

   We  can  write therefore
   $$\Theta(J_1)= \Theta(\check J)\ast \Theta( \{ m,\dots, k_{n-1},n-1\} ) $$
   $$\Theta(J_2)= \Theta(\check J)\ast \Theta( \{ m,\dots, a, k_{n-1}  \}). $$
 Now, we  obtain  $\Theta(k_{n-1})=n-3$, and  $\Theta(a)=n-2$ in  both  cases.
 Therefore,  since  $\Theta$  does  not  touch  $n-1$,
   $$\Theta(J_1)= \Theta(\check J)\ast  \{  \Theta(m),\dots, n-3,n-1\} $$
   $$\Theta(J_2)= \Theta(\check J)\ast  \{ \Theta(m),\dots,  n-3,n-2. \}.$$
Now,  we  have
$$\varrho_{n-1}(T_{n-2}E_{J'_i})=\A \, E_{\tau_{n-1,n-2}(J'_i)} $$
where,  for  both  $i=1$  and  $i=2$,  we have
$$\tau_{n-1,n-2}(J'_i)=(\Theta(J_i)\ast\{(n-1),(n-2)\})\backslash (n-1)=(\Theta(\check J)\ast  \{  \Theta(m),\dots,  n-3,n-2,n-1\})\backslash (n-1).$$
Therefore,  $\varrho_{n-1}(T_{n-2}E_{J'_1})=\varrho_{n-1}(T_{n-2}E_{J'_2})$.


\smallbreak

 We will prove now (ii). Firstly, we   study    the  cases when $k_{n-1}=0$ or $k_{n-2}=0$. In the case $k_{n-1}=0$, we have:
 $$
T_{n-1}X =  \mathbb{T}_{1,k_1}  \mathbb{T}_{2,k_2}\cdots \mathbb{T}_{n-1,k_{n-2}}E_J\quad
\text{and}\quad XT_{n-1} = \mathbb{T}_{1,k_1}  \mathbb{T}_{2,k_2}\cdots \mathbb{T}_{n-2,k_{n-2}}T_{n-1}E_{s_{n-1}(J)}.
 $$
 Then,
 $$
 \varrho_n(T_{n-1}X) = \mathsf{A} \mathbb{T}_{1,k_1}  \mathbb{T}_{2,k_2}\cdots \check{\mathbb{T}}_{n-1,k_{n-2}}E_{\tau_{n-1,k_{n-2}}(J)}
 $$
 and
 $$
\varrho_n(XT_{n-1}) = \mathsf{A}\mathbb{T}_{1,k_1}  \mathbb{T}_{2,k_2}\cdots \mathbb{T}_{n-2,k_{n-2}}E_{\tau_{n,n-1}(J)}.
 $$
Now
  $$ \varrho_{n-1}(\varrho_n(T_{n-1}X))=\A^2 \, \mathbb{T}_{1,k_1}  \mathbb{T}_{2,k_2}\cdots \check{\mathbb{T}}_{n-2,k_{n-2}}E_{\tau_{n-1,k_{n-2}}(\tau_{n,k_{n-2}}(J))}$$
and
  $$ \varrho_{n-1}(\varrho_n(XT_{n-1})) =\A^2 \, \mathbb{T}_{1,k_1}  \mathbb{T}_{2,k_2}\cdots \check{\mathbb{T}}_{n-2,k_{n-2}}E_{\tau_{n-1,k_{n-2}}(\tau_{n-1,k_{n-2}}(J))}.$$
But
$$\tau_{n-1,k_{n-2}}(\tau_{n,k_{n-2}}(J))=((J \ast \{k_{n-2},n\})\backslash n)\ast   \{k_{n-2},n-1\}) \backslash (n-1)$$  and
$$\tau_{n-1,k_{n-2}}(\tau_{n, n-1}(J))=((J \ast \{ n-1,n\}) \backslash n)\ast \{k_{n-2},n-1\})   \backslash (n-1).$$
The right  members of  the  preceding  two  equalities  are both equal  to
$$ ((J \ast \{k_{n-2}, n-1,n\}) \backslash n)   \backslash (n-1)$$  so  that the  proof is  completed.

  For the case $k_{n-2}=0$, we have:
$$
T_{n-1}X =  \mathbb{T}_{1,k_1}  \mathbb{T}_{2,k_2}\cdots \mathbb{T}_{n-3,k_{n-3}}T_{n-1}\mathbb{T}_{n-1,k_{n-1}}E_J
$$
and
$$
 XT_{n-1} = \mathbb{T}_{1,k_1}  \mathbb{T}_{2,k_2}\cdots \mathbb{T}_{n-3,k_{n-3}}\mathbb{T}_{n-1,k_{n-1}}T_{n-1}
 E_{s_{n-1}(J)}.
 $$
 By using Lemma \ref{relativepropertie1} we  get  that
 $\varrho_{n-1}(\varrho_n(T_{n-1}X))$ and $\varrho_{n-1}(\varrho_n(XT_{n-1}))$  are different,
 respectively, in $$R:=\varrho_{n-1}(\varrho_n( T_{n-1}\mathbb{T}_{n-1,k_{n-1}}E_J))  \ \text{ and} \
 S:=\varrho_{n-1}(\varrho_n( \mathbb{T}_{n-1,k_{n-1}}T_{n-1} E_{s_{n-1}(J)}).$$  It is a routine to check that these two last expression are equal for $k_{n-1}=0,n-1$. Thus, we need to check only that $R=S$  for $0<k_{n-1}<n-1$. Now,
 \begin{eqnarray*}
 T_{n-1}\mathbb{T}_{n-1,k_{n-1}}E_J & =&  T_{n-1}^2T_{n-2}\mathbb{T}_{n-3,k_{n-1}}E_J\\
 &= &
 (T_{n-2} + (u-1)E_{n-1}T_{n-2} + (u-1)E_{n-1}T_{n-1}T_{n-2})\mathbb{T}_{n-3, k_{n-1}}E_J\\
 &= &
 (T_{n-2} + (u-1)E_{n-1}T_{n-2} + (u-1)E_{n-1}T_{n-1}T_{n-2})E_{\theta_{n-3, k_{n-1}}(J)}\mathbb{T}_{n-3, k_{n-1}}
 \end{eqnarray*}
Let's  us  call  $\Theta:=\theta_{n-3, k_{n-1}}$
Then, by using again Lemma \ref{relativepropertie1}, we obtain
$$
R = (R_1 + (u-1)R_2 +(u-1)R_3)\mathbb{T}_{n-3, k_{n-1}}
$$
where
\begin{eqnarray*}
R_1 &:= &
\varrho_{n-1}(\varrho_{n}(T_{n-2}E_{\Theta(J)}))\\
R_2 & := &
\varrho_{n-1}(\varrho_{n}(E_{n-1}T_{n-2}E_{\Theta(J)}))\\
R_3 & := &
\varrho_{n-1}(\varrho_{n}(E_{n-1}T_{n-1}T_{n-2}E_{\Theta(J)})).
\end{eqnarray*}
Now,  we  have
\begin{eqnarray*}
R_1 &= &
\B \varrho_{n-1}(  T_{n-2}E_{ \Theta(J)\backslash n})= \mathsf{AB}E_{J^R_1} \\
R_2 & = &
 \B \varrho_{n-1}( T_{n-2}E_{(\{n-2,n\}\ast \Theta(J))\backslash n})=\A \B E_{J^R_2}  \\
R_3 & = &
\A \varrho_{n-1}( T_{n-2}E_{\tau_{n,n-2}(\{n-2,n\} \ast \Theta(J))})= \A^2 E_{J^R_3}
 \end{eqnarray*}
where
\begin{eqnarray*}
J^R_1  &  = & ((\Theta(J)\backslash n )\ast\{n-2,n-1\})\backslash (n-1)  \\
  J^R_2 &  = &
    (((\{n-2,n\}\ast \Theta(J))\backslash n )\ast\{n-2,n-1\})\backslash (n-1) \\
  J^R_3 & = &\tau_{n-1,n-2}( \tau_{n,n-2}(\{n-2,n\} \ast \Theta(J))=
   J^R_2.
\end{eqnarray*}

 On the other part,
 \begin{eqnarray*}
 \mathbb{T}_{n-1,k_{n-1}}T_{n-1} E_{s_{n-1}(J)}
 & = &
 T_{n-1}T_{n-2}T_{n-1}\mathbb{T}_{n-3, k_{n-1}}E_{s_{n-1}(J)} \\
 & = &
  T_{n-2}T_{n-1}T_{n-2}E_{\Theta (s_{n-1}J)} \mathbb{T}_{n-3, k_{n-1}}.
 \end{eqnarray*}
 Then, again from Lemma \ref{relativepropertie1}, we  get
\begin{eqnarray*}
S & = &   \varrho_{n-1}(\varrho_{n}( T_{n-2}T_{n-1}T_{n-2}E_{\Theta(s_{n-1} J)} ))\mathbb{T}_{n-3, k_{n-1}}\\
 & = &
\mathsf{A}\varrho_{n-1}(T_{n-2}^2 E_{\tau_{n,n-2}(\Theta(s_{n-1} J)))} )\mathbb{T}_{n-3, k_{n-1}}\\
& =& (S_1 + (u-1)S_2 +(u-1) S_3)\mathbb{T}_{n-3, k_{n-1}}
\end{eqnarray*}
where
\begin{eqnarray*}
S_1 & := &
 \mathsf{A}\varrho_{n-1}( E_{J^S })= \A \B E_{J^S_1}\\
S_2 & := &
 \mathsf{A}\varrho_{n-1}(E_{n-2} E_{J^S}) = \A \B E_{J^S_2}\\
S_3 & := &
\mathsf{A}\varrho_{n-1}(T_{n-2}E_{n-2} E_{J^S}) =  \A^2 E_{J^S_3}
\end{eqnarray*}
being
\begin{eqnarray*}
J^S  & = &\tau_{n,n-2}(\Theta(s_{n-1} J ))= (\Theta(s_{n-1} J )\ast\{n-2,n\} ) \backslash n,\\
  J^S_1 & = &J^S\backslash (n-1)=((\Theta(s_{n-1} J )\ast\{n-2,n\} ) \backslash n)\backslash (n-1),\\
 J^S_2& = & (J^S \ast \{n-2,n-1\})\backslash(n-1)= (((\Theta(s_{n-1} J )\ast\{n-2,n\} )\backslash n) \ast\{n-2,n-1\})   \backslash (n-1),\\
 J^S_3& = &\tau_{n-1,n-2}(J^S \ast \{n-2,n-1\})=(  (  (\Theta(s_{n-1} J )\ast\{n-2,n\})\backslash n)  \ast\{n-2,n-1\})  \backslash (n-1)= J^S_2.
 \end{eqnarray*}

Now,  observe  that
$$(\Theta(s_{n-1} J )\ast\{n-2,n\} ) \backslash n=  ((\Theta(J)\backslash n )\ast\{n-2,n-1\}) $$
since  $\Theta$  does  not  touch  $n-1,n$.
 Thus, we have that for  $i=1,2,3$,  $J^R_i=J^S_i$,  and  therefore  also $R_i= S_i$.  The  proof  is  concluded.

In order to finish the proof of (ii), we have to check the claim in  the  cases $k_{n-1}$ and  $k_{n-2}$ both different from  $0$.  We will compute first
 $\varrho_{n-1}(\varrho_n(T_{n-1}X))$:
$$
 T_{n-1}X
 =
 \mathbb{T}_{1,k_1}\cdots  \mathbb{T}_{n-3,k_{n-3}}(
T_{n-1}\mathbb{T}_{n-2,k_{n-3}} \mathbb{T}_{n-1,k_{n-3}}E_{J}).
$$
 Then, from Lemma \ref{relativepropertie1}:
$$
\varrho_{n-1}(\varrho_n(T_{n-1}X)) =
\mathbb{T}_{1,k_1}\cdots  \mathbb{T}_{n-3,k_{n-3}}G
$$
where $G:= \varrho_{n-1}(\varrho_{n}(T_{n-1}\mathbb{T}_{n-2,k_{n-2}} \mathbb{T}_{n-1,k_{n-2}}E_J))$.

We   compute now  $\varrho_{n-1}(\varrho_n(XT_{n-1}))$. From (\ref{TE_I}) we have
$$
XT_{n-1} =  \mathbb{T}_{1,k_1}  \cdots \mathbb{T}_{n-1,k_{n-1}}T_{n-1} E_{s_{n-1}(J)}.
$$
  Lemma \ref{relativepropertie1}  implies that
$$
\varrho_{n-1}(\varrho_{n}(XT_{n-1})) =  \mathbb{T}_{1,k_1}  \cdots \mathbb{T}_{n-3,k_{n-3}}H
$$
where $H:= \varrho_{n-1}(\varrho_{n}( \mathbb{T}_{n-2,k_{n-2}}\mathbb{T}_{n-1,k_{n-1}}T_{n-1} E_{s_{n-1}(J)}))$. Thus, it is enough to prove that $G=H$. We will do this by
distinguishing   four cases,  according  to  the  values of  $k_{n-1}$ and $k_{n-2}$.

  Case $k_{n-1}= n-1$ and $k_{n-2}= n-2$. In this case we have:
$ \mathbb{T}_{n-2,k_{n-2}}=T_{n-2}$ and  $\mathbb{T}_{n-1,k_{n-1}} = T_{n-1}$. We have then
$$
G= \varrho_{n-1}(\varrho_{n}(T_{n-1}T_{n-2} T_{n-1}E_J)) \quad \text{and} \quad
H=\varrho_{n-1}(\varrho_{n}( T_{n-2}T_{n-1}T_{n-1} E_{s_{n-1}(J)})).
$$
  $G$ can be  rewritten  as
\begin{eqnarray*}
G
& = &
\varrho_{n-1}(\varrho_{n}(T_{n-2}T_{n-1} T_{n-2}E_J))\\
& = &
\mathsf{A}\varrho_{n-1}(T_{n-2}^2 E_{J'})\\
& = &
\mathsf{A}\varrho_{n-1}((1 + (u-1)E_{n-2} + (u-1)E_{n-2}T_{n-2}) E_{J'})\\
& = &
G_1 + (u-1)G_2 +(u-1) G_3
\end{eqnarray*}
where $J'=(J \ast \{n,n-2\})\backslash n$  and
\begin{eqnarray*}
G_1 & := &
\mathsf{A}\varrho_{n-1}(E_{J'}) = \mathsf{AB}E_{J'\backslash (n-1)}\\
G_2 & := &
\mathsf{A}\varrho_{n-1}(E_{n-2}E_J)= \mathsf{AB}E_{J'\ast\{n-2, n-1\}\backslash (n-1)}\\
G_3 & := &
\mathsf{A}\varrho_{n-1}(T_{n-2}E_{n-2} E_J)=\mathsf{A^2}E_{J'\ast\{n-2, n-1\}\backslash (n-1)}.
\end{eqnarray*}
In order to compute  $H$, we   firstly note  that
$$
T_{n-2}T_{n-1}T_{n-1} E_{s_{n-1}(J)} = T_{n-2}( 1 + (u-1) E_{n-1} + (u-1)T_{n-1}E_{n-1})E_{s_{n-1}(J)}.
$$
Then,
$$
H = H_1 + (u-1) H_2 + (u-1)H_3
$$
where
\begin{eqnarray*}
H_1
& := &
\varrho_{n-1}(\varrho_n(T_{n-2}E_{s_{n-1}(J)}))= \mathsf{B}\varrho_{n-1}(T_{n-2}E_{s_{n-1}(J)\backslash n }) = \mathsf{AB}E_{s_{n-1}(J)\backslash n },\\
H_2
 & := &
\varrho_{n-1}(\varrho_n( T_{n-2}E_{n-1}E_{s_{n-1}(J)}))= \mathsf{B}\varrho_{n-1}( T_{n-2}E_{s_{n-1}(J)\ast\{n-1, n\}\backslash n})\\
& = &
 \mathsf{AB}E_{(((J \ast\{n-1, n\})\backslash n)\ast\{n-2, n-1\})\backslash (n-1)  }, \\
H_3
& := &
\varrho_{n-1}(\varrho_n(  T_{n-2}T_{n-1}E_{n-1}E_{s_{n-1}(J)})) =
\mathsf{A}\varrho_{n-1}(T_{n-2}E_{ (J\ast\{n-1, n\}) \backslash n})\\
& = &
\mathsf{A}^2E_{((( J \ast\{n-1, n\})\backslash n)\ast\{n-2,n-1\})\backslash (n-1).}
\end{eqnarray*}
Thus the  equality $G=H$ is a consequence of the  equalities $G_i=H_i$,  $i=1,2,3$.

We  will  analyze now  the  remaining  cases  $0<k_{n-1}< n-1$ and $0<k_{n-2}< n-2$.

Observe that
\begin{eqnarray*}
T_{n-1}\mathbb{T}_{n-2,k_{n-2}} \mathbb{T}_{n-1,k_{n-1}}
& = &
T_{n-1}(T_{n-2}T_{n-1}T_{n-3}\cdots T_{k_{n-2}} )\check{\mathbb{T}}_{n-1,k_{n-1}}\\
& = &
T_{n-2}T_{n-1}T_{n-2}T_{n-3}\cdots T_{k_{n-2}} \check{\mathbb{T}}_{n-1,k_{n-1}}.
\end{eqnarray*}
Therefore,
\begin{equation}
T_{n-1}\mathbb{T}_{n-2,k_{n-2}} \mathbb{T}_{n-1,k_{n-2}}E_J =
T_{n-2}T_{n-1}E_{J^{\prime}}\mathbb{T}_{n-2,k_{n-2}}\check{\mathbb{T}}_{n-1,k_{n-1}}
\end{equation}
where $J^{\prime}:=\theta^{-1}_{n-2,k_{n-1}}\theta^{-1}_{n-2,k_{n-2}}(J) $. Thus, by using Lemma \ref{relativepropertie1}, we get
\begin{eqnarray*}
G
& = &
\varrho_{n-1}(T_{n-2} \varrho_{n}(T_{n-1}E_{J^{\prime}})\mathbb{T}_{n-2,k_{n-2}}\check{\mathbb{T}}_{n-1,k_{n-1}})
= \mathsf{A} G^{\prime}
\end{eqnarray*}
where $G^{\prime}:= \varrho_{n-1}(T_{n-2} E_{\tau_{n,n-1}(J^{\prime})}\mathbb{T}_{n-2,k_{n-2}}\check{\mathbb{T}}_{n-1,k_{n-1}})$.
Now, we have
$$
T_{n-2} E_{ \tau_{n,n-1}(J^{\prime})}\mathbb{T}_{n-2,k_{n-2}}\check{\mathbb{T}}_{n-1,k_{n-1}}=
E_{J^{\prime\prime}}T_{n-2} \mathbb{T}_{n-2,k_{n-2}}\check{\mathbb{T}}_{n-1,k_{n-1}},
$$
where $J^{\prime\prime}:=s_{n-2} (\tau_{n,n-1}(J^{\prime}))$. Then
\begin{eqnarray*}
E_{J^{\prime\prime}} T_{n-2} \mathbb{T}_{n-2,k_{n-2}}\check{\mathbb{T}}_{n-1,k_{n-1}}
& = &
E_{J^{\prime\prime}} T_{n-2}(T_{n-2}T_{n-3}T_{n-2}\cdots T_{k_{n-2}})\check{\check{\mathbb{T}}}_{n-1,k_{n-1}} \\
& = &
E_{J^{\prime\prime}} T_{n-2}T_{n-3}T_{n-2}T_{n-3}\cdots T_{k_{n-2}}\check{\check{\mathbb{T}}}_{n-1,k_{n-1}}\\
& = &
E_{J^{\prime\prime}} T_{n-3}T_{n-2}T_{n-3}T_{n-3}\cdots T_{k_{n-2}}\check{\check{\mathbb{T}}}_{n-1,k_{n-1}}\\
& = &
 T_{n-3}(T_{n-2}E_{J^{\prime\prime\prime}})T_{n-3}T_{n-3}\cdots T_{k_{n-2}}\check{\check{\mathbb{T}}}_{n-1,k_{n-1}}
\end{eqnarray*}
where  $ J^{\prime\prime\prime}:= s_{n-2}s_{n-3}J^{\prime\prime}=(n-3,n-1)(\tau_{n,n-1}(J'))$. Therefore
$$
G^{\prime} = \mathsf{A}\, T_{n-3}E_{\tau_{n-1,n-2}(J^{\prime\prime\prime})}T_{n-3}T_{n-3}\mathbb{T}_{n-4, k_{n-2}}\check{\check{\mathbb{T}}}_{n-1,k_{n-1}}.
$$

In $H$, we have
\begin{eqnarray*}
H
& = &
\varrho_{n-1}(\varrho_{n}( \mathbb{T}_{n-2,k_{n-2}}T_{n-2}\mathbb{T}_{n-1,k_{n-1}}E_{s_{n-1}(J)}))\quad (\text{from (\ref{Tt}}))\\
& = &
\varrho_{n-1}( \mathbb{T}_{n-2,k_{n-2}}T_{n-2}\varrho_{n}(\mathbb{T}_{n-1,k_{n-1}}E_{s_{n-1}(J)}))  \quad (\text{from Lemma \ref{relativepropertie1}}) \\
& = & \mathsf{A}H^{\prime}
\end{eqnarray*}
where $H^{\prime}: =\varrho_{n-1}( \mathbb{T}_{n-2,k_{n-2}}T_{n-2}\check{\mathbb{T}}_{n-1,k_{n-1}}
E_{\tau_{n,k_{n-1}}(s_{n-1} J )})$. Observe now that
\begin{eqnarray*}
\mathbb{T}_{n-2,k_{n-2}}T_{n-2}\check{\mathbb{T}}_{n-1,k_{n-1}}
& = &
T_{n-2}T_{n-3}T_{n-2}T_{n-2}\mathbb{T}_{n-4, k_{n-2}}\check{\check{\mathbb{T}}}_{n-1,k_{n-1}}\\
& = & T_{n-3}T_{n-3}T_{n-2}
T_{n-3}\mathbb{T}_{n-4, k_{n-2}}\check{\check{\mathbb{T}}}_{n-1,k_{n-1}}.
\end{eqnarray*}
Then,
$$
\mathbb{T}_{n-2,k_{n-2}}T_{n-2}\check{\mathbb{T}}_{n-1,k_{n-1}}
E_{\tau_{n,k_{n-1}}(s_{n-1} J )}  = T_{n-3}T_{n-3}T_{n-2}E_{I}
T_{n-3}\mathbb{T}_{n-4, k_{n-2}}\check{\check{\mathbb{T}}}_{n-1,k_{n-1}}
$$
where $I := \theta_{n-3,k_{n-2}} \theta_{n-3, k_{n-1}}(\tau_{n,k_{n-1}}(s_{n-1}J ))$. Thus
\begin{eqnarray*}
H^{\prime} &  =  &  T_{n-3}T_{n-3}\varrho_{n-1}(T_{n-2}E_{I})
T_{n-3}\mathbb{T}_{n-4, k_{n-2}}\check{\check{\mathbb{T}}}_{n-1,k_{n-1}} \\
& = &  \mathsf{A}\, T_{n-3}T_{n-3}E_{\tau_{n-1,n-2}(I)}
T_{n-3}\mathbb{T}_{n-4, k_{n-2}}\check{\check{\mathbb{T}}}_{n-1,k_{n-1}}.
\end{eqnarray*}
Therefore, to have  $G^{\prime} = H^{\prime}$ and then $G=H$, it is enough to prove that
$$
T_{n-3}E_{\tau_{n-1,n-2}(J^{\prime\prime\prime})}T_{n-3}T_{n-3}
=
T_{n-3}T_{n-3}E_{\tau_{n-1,n-2}(I)}T_{n-3}
$$
i.e.  that
$$ \tau_{n-1,n-2}(J^{\prime\prime\prime})=s_{n-3}\tau_{n-1,n-2}(I).$$
The  left  member  is  equal  to
\begin{equation}\label{J1}
    \tau_{n-1,n-2}((n-3,n-1) \tau_{n,n-1}(\theta^{-1}_{n-2,k_{n-1}}\theta^{-1}_{n-2,k_{n-2}}(J)) ) \end{equation}
while  the  right  member is  equal  to
\begin{equation}\label{J2}
s_{n-3}\tau_{n-1,n-2}(\theta_{n-3,k_{n-2}} \theta_{n-3, k_{n-1}}(\tau_{n,k_{n-1}}(s_{n-1}J ))).   \end{equation}
Observe that  (\ref{J1})=(\ref{J2}) in  the  extreme situations  when $J=(\ )=(\{1\},\{2\},\dots,\{n\}\})$  and $J=(\{1,2,\dots,n\})$. In both these  cases
  (\ref{J1}) and (\ref{J2}) are  given  by  $(J\backslash n)\backslash (n-1)$.
Otherwise,  we  have  to distinguish  the  cases  $k_{n-2}<k_{n-1}$  and
   $k_{n-2}\ge k_{n-1}$,  and  the proof is  done  by  comparing  the  set--partitions.   We prefer to  avoid  two further   boring pages of  calculations, using  the  same  arguments  as in the  proof of point (i). Let's give  a non  trivial  example.
Let $n=7$,  $k_{6}=3$,  $k_{5}=1$ and  $J=(\{1,2\},\{3\},\{5,7\},\{4,6\})$. We  calculate  (\ref{J1}):
\begin{eqnarray*}
  J &=& (\{1,2\},\{3\},\{5,7\},\{4,6\}) \\
 J'=\theta^{-1}_{5,3}\theta^{-1}_{5,1}(J)  &=& (\{2,4\},\{5\},\{3,7\},\{6,1\}) \\
 J''= \tau_{7,6}(J')  &=&  (\{2,4\},\{5\},\{3, 6,1\})  \\
  J'''=(4,6)(J'') &=&  (\{2,6\},\{5\},\{3, 4,1\}) \\
  \tau_{6,5}(J''') &=&  (\{2, 5\},\{3, 4,1\}).
\end{eqnarray*}

As  for    (\ref{J2}), we  have:
 \begin{eqnarray*}
  J &=& (\{1,2\},\{3\},\{5,7\},\{4,6\}) \\
 J'= s_6(J) &=& (\{1,2\},\{3\},\{5,6\},\{4,7\}) \\
 J''= \tau_{7,3}(J')  &=&  (\{1,2\},\{5,6\}, \{3, 4\})  \\
  J'''=\theta^{-1}_{4,3}\theta^{-1}_{4,1} (J'') &=&  (\{5,1\},\{3,6\},\{2,4\}) \\
  \tilde J=\tau_{6,5}(J''') &=&  (\{5,1,3\},\{2, 4\})\\
 s_4(  \tilde J) &=&  (\{4,1,3\},\{2, 5\}).
\end{eqnarray*}
The   claim (iii) follows  directly from   claims (i) and (ii).

\end{proof}





\subsection{} \label{sectionmarkovtrace2}

Let us   define  $\varrho_1$ as the  identity  and for  every  positive integer  we define the  linear map $\rho_n: \mathcal{E}_n \rightarrow \mathcal{E}_1 = K(\mathsf{A}, \mathsf{B})$ by
$$
\rho_n =\varrho_1 \circ  \varrho_2 \circ  \cdots  \circ \varrho_{n-1}\circ \varrho_n.
$$
Observe  that
$$  \rho_n= \rho_{n-1} \circ \varrho_n,$$
and  that, for $k\leq n$ and $X\in \mathcal{E}_k$,
 \begin{equation}
\rho_n(X) = \rho_k(X).
\end{equation}
Also,  from the definition of $\varrho_n$, it follows that $\rho_n(1) =1$. Moreover, we have the following theorem.
\begin{theorem}\label{trace}
The family   $\rho := \{\rho_n\}_{n\ge 1}  $ is a Markov trace.  That is, for every  $n\ge 1$    $\rho_n$  has the following  properties:
\begin{enumerate}
\item[(i)] $\rho_n(1) = 1$
\item[(ii)] $\rho_n(XY) = \rho_n(YX)$
\item[(iii)] $\rho_{n+1}(XT_n) =\rho_{n+1}(XE_nT_n)=\mathsf{A} \rho_n(X)$
\item[(iv)] $\rho_{n+1}(XE_n) =\mathsf{B} \rho_n(X)$
\end{enumerate}
where  $X, Y\in {\mathcal E}_n$.
\end{theorem}
\begin{proof}

We  will prove (ii)  by induction on $n$. For $n=2$ clearly the claim is true
since $\mathcal{E}_2$ is commutative. Suppose now the claim  be true
for all $k$ less than $n$.

 Firstly, we are going  to prove (ii) for $X\in \mathcal{E}_n$ and
$Y\in \mathcal{E}_{n-1}$. We have
\begin{eqnarray*}
\rho_n(XY) & = &  \rho_{n-1}(\varrho_n(XY))\\
 & =&
 \rho_{n-1}  (\varrho_n(X)Y) \quad \text{ ((i) Lemma
\ref{relativepropertie1})}\\
  & =&
 \rho_{n-1}(Y\varrho_n(X)) \quad \text{(induction hypothesis)}\\
 & =&
 \rho_{n-1}(\varrho_n(YX))\quad \text{ ((ii) Lemma \ref{relativepropertie1}).}
 \end{eqnarray*}
 Hence,
\begin{equation}\label{proof1}
  \rho_n(XY) = \rho_n(YX) \qquad (X\in \mathcal{E}_{n}, Y\in \mathcal{E}_{n-1}).
 \end{equation}
 Secondly, we  prove (ii) for $Y\in \{T_{n-1}, \, E_{n-1},
T_{n-1} E_{n-1}\}$. We have
$\rho_n(XY)  =\rho_{n-2}(\varrho_{n-1}(\varrho_n(XY)))$, then
from     Lemma \ref{relativepropertie3}, we deduce $\rho_n(XY)  =
 \rho_{n-2} (\varrho_{n-1}(\varrho_n(YX)))$. Hence
\begin{equation}\label{proof2}
\rho_n (XY)= \rho_n(YX)\qquad (X\in \mathcal{E}_{n}, Y\in \{T_{n-1},
\, E_{n-1}, T_{n-1} E_{n-1}\}).
 \end{equation}
Now, having in mind   Corollary \ref{atmost} and the linearity of
$\rho_n$,   to prove  claim (ii) it is enough to
consider   $Y$ in the form
$Y_1FY_2$, where $Y_1, Y_2\in \mathcal{E}_{n-1}$ and $F\in \{T_{n-1},
E_{n-1}, T_{n-1}E_{n-1}\}$. Indeed,  if  $\rho_n(XY)= \rho_n(XY_1FY_2)$,
then, by using (\ref{proof1}), we have
$\rho_n(XY)= \rho_n(Y_2XY_1F)$.   By using now (\ref{proof2}), we
obtain $\rho_n(XY)= \rho_n(FY_2XY_1)$. Thus, by using again
(\ref{proof1}), we  get  $\rho_n(XY) = \rho_n(Y_1FY_2X)$. Hence,
$\rho_n(XY) = \rho_n(YX)$.

The proofs of the statements  (iii) and (iv) are   analogous. We shall prove only that, if $X\in \E_n$,  then
$\rho_{n+1}(XT_n) = \mathsf{A} \rho_n(X)$. We have:
\begin{eqnarray*}
 \rho_{n+1}(XT_n) & = &   \rho_n(\varrho_{n+1}(XT_n))\\
 & =&
 \rho_n(X\varrho_{n+1}(T_n)) \quad \text{ ((ii) Lemma \ref{relativepropertie1})}\\
  & =&
 \rho_n(X \mathsf{A}  ) \\
  & =&
   \mathsf{A} \rho_n(X).
 \end{eqnarray*}
\end{proof}


\begin{remark}\label{Econdition}\rm
 Observe  that rule (iv) in the above theorem  is the  condition on the Markov trace of the Yokonuma--Hecke algebra requested  to have the  invariant defined by S. Lambropoulou and the second author, see \cite{julaAM, julaJKTR, julaMKTA}. More precisely,
this property allows to factorize    $\rho_n(X)$ in the computation of $\rho_{n+1}(XT_n^{-1})$, where $X\in \mathcal{E}_n$, as  we  will show in the  following section (see formula (\ref{factorize})).
\end{remark}


\section{Applications to Knot invariants}\label{sectionapplications}
In this section we will construct  an invariant for classical knots and another invariant  for singular knots.
The constructions follow the Jones' recipe, that is, they are obtained from normalization and rescaling
of the composition of a representation of a braid group/singular braid monoid in $\mathcal{E}_n$ with the trace $\rho_n$.

In both invariants we will use the element of normalization $\mathsf{L}= \mathsf{L}(u, \mathsf{A}, \mathsf{B})$, defined as follows
\begin{equation}\label{definionL}
\mathsf{L} = \frac{\mathsf{A} + (1-u)\mathsf{B}}{u\mathsf{A}}, \quad \text{or  equivalently} \quad \mathsf{A} = -\frac{1-u}{1-\mathsf{L}u}\B.
\end{equation}

\subsection{}
In order to define our invariant for classical knots, we  recall  some classical facts and standard notation.
Firstly,
remember  that according to  the classical theorems of Alexander and Markov,   the set of  isotopy classes of links in
the Euclidean space is in  bijection with the set of equivalence  classes obtained from  the inductive limit
 of the tower of
braid groups $B_1\subseteq B_2
\subseteq \cdots \subseteq B_n \subseteq\cdots$, under the {\it Markov equivalence relation} $\sim$. That is, for all $\alpha , \beta\in B_n$, we have:
\begin{enumerate}
 \item[(i)] $\alpha\beta \sim \beta\alpha$
 \item[(ii)] $\alpha\sim \alpha \sigma_n $ and  $\alpha\sim \alpha \sigma_n^{-1} $.
\end{enumerate}
Secondly, let us denote  $\bar{\pi}_{\mathsf{L} }$ the  representation of $B_n$ in $\mathcal{E}_n$, namely $\sigma_i \mapsto \sqrt{\mathsf{L}} T_i$. Then, for $\alpha\in B_n$, we define
$\bar{\Delta}(\alpha) $
\begin{equation}\label{Deltabarra}
\bar{\Delta}(\alpha) :=  \left(-\frac{1-\mathsf{L}u}{\sqrt{\mathsf{L}}(1-u)\B}\right)^{n-1}(\rho_n\circ \bar{\pi}_{\mathsf{L} } )(\alpha)\in K(\sqrt{\mathsf{L}}, B).
\end{equation}
It is useful to have an alternative expression for  $\bar{\Delta}(\alpha)$, in terms of the exponent $e(\sigma)$ of $\alpha$, where
  $e(\alpha)$ is the algebraic sum of the exponents of the elementary braids $\sigma_i$ used  for writing  $\alpha$. Then, we have:

\smallbreak
\begin{equation}\label{jones}
\bar{\Delta}(\alpha) = \left(-\frac{1-\mathsf{L}u}{\sqrt{\mathsf{L}}(1-u)\B}\right)^{n-1}
(\sqrt{\mathsf{L}})^{e(\alpha)}
(\rho_n\circ \bar{\pi} )(\alpha)
\end{equation}
 where $\bar{\pi}$ is defined as  the  mapping $\sigma_i\mapsto T_i$. Now, we have

\begin{equation}\label{LDA=1}
\sqrt{\mathsf{L}}\, \bar{D}\, \mathsf{A}= 1,  \quad \text{where} \quad \bar{D} := -\frac{1-\mathsf{L}u}{\sqrt{\mathsf{L}}(1-u)\B}.
\end{equation}

Then, notice that  $\bar{\Delta}(\alpha)$ can be rewritten as follows:
$$
\bar{\Delta}(\alpha) = \bar{D}^{n-1} (\sqrt{\mathsf{L}})^{e(\alpha )} (\rho_n\circ \bar{\pi} )(\alpha).
$$
\begin{theorem}\label{Delta}
 Let  $L$  be a link  obtained by  closing the braid $\alpha\in B_n$. Then the map $L\mapsto \bar{\Delta}(\alpha) $ defines
 an isotopy invariant of links.
\end{theorem}
\begin{proof}
It is enough to prove that $\bar{\Delta}$ respects the Markov equivalence relations. In virtue of   Theorem \ref{trace} (ii), it is evident that $\bar{\Delta}$ respects the first Markov equivalence. We have  to   prove    the    second Markov equivalence.  Again, it is  easy to check that $\bar{\Delta}(\alpha) = \bar{\Delta} (\alpha \sigma_n )$. In fact, up to
 now
 we have only used the properties of the trace $\rho_n$,  in   which  the elements   $E_i$'s do not play  any role.  But now,  to prove that
  $\bar{\Delta}(\alpha)=\bar{\Delta}(\alpha \sigma_n^{-1})$,   the defining conditions of $\rho_n$ involving the elements $E_i$'s are crucial (see Remark \ref{Econdition}).

For every $\alpha\in B_n$ we have
\begin{eqnarray*}
 \bar{\Delta}(\alpha\sigma_n^{-1})
 & =   &
 \bar{D}^{n}(\sqrt{\mathsf{L}})^{e(\alpha \sigma_n^{-1})} (\rho_n( \bar{\pi} (\alpha\sigma_n^{-1}))\\
  & =   &
 \bar{D}^{n}(\sqrt{\mathsf{L}})^{e(\alpha) -1 } (\rho_n(\bar{\pi} (\alpha) T_n^{-1})).
\end{eqnarray*}
By using the formulae of $T^{-1}_n$ (see Proposition \ref{relations}) and the defining rule of $\rho_n$, we deduce:
\begin{equation}\label{factorize}
\rho_n(\bar{\pi}(\alpha) T_n^{-1}) = \rho_n(\bar{\pi} (\alpha))(  \mathsf{A} + (u^{-1}-1)\mathsf{B} + (u^{-1}-1)  \mathsf{A}).
\end{equation}
Then
\begin{eqnarray*}
\bar{\Delta}(\alpha\sigma_n^{-1})
& = &
\bar{D}^{n}(\sqrt{\mathsf{L}})^{e(\alpha) -1 } ( u^{-1} \mathsf{A} + (u^{-1}-1)\mathsf{B})
\rho_n(\bar{\pi} (\alpha))\\
& = &
 (\bar{D}/ \sqrt{\mathsf{L}})
 ( u^{-1} \mathsf{A} + (u^{-1}-1)\mathsf{B})\bar{D}^{n-1}(\sqrt{\mathsf{L}})^{e(\alpha) }
\rho_n(\bar{\pi} (\alpha))\\
& =&
(\bar{D}/ \sqrt{\mathsf{L}})
\mathsf{A}\mathsf{L}\bar{D}^{n-1}(\sqrt{\mathsf{L}})^{e(\alpha) } \rho_n(\bar{\pi} (\alpha))\\
& = &
\bar{\Delta}(\alpha)\qquad \text{(by (\ref{LDA=1}))}.
\end{eqnarray*}
\end{proof}

\begin{example} Let  $\alpha$ be   the simplest oriented link,  formed  by  two oriented  circles  with  two  positive  crossings. We  obtain
$$\rho_n(\bar{\pi} (\alpha))= 1+(\A+\B)(u-1) $$
and
$$\bar{\Delta}(\alpha)=  \sqrt{\frac{\A+\B(1-u)}{u\A}}\left(\frac{1+(\A+\B)(u-1)}{\A}\right).$$
\end{example}
\begin{example}
Let  $\gamma$  be   the  trefoil knot with  positive  crossings. We  obtain
$$\rho_n(\bar{\pi} (\gamma))= \frac{\B(1-u+u^2-u^3)+\A(1-u+u^2)}{u^3} $$
and
$$\bar{\Delta}(\gamma)= \frac{\A(-u^3\B+u^2\B-u\B+\B+u^2\A-u\A+\A)}{u(\A+\B-u\B)^2}.$$
 \end{example}

\subsection{}
 For the singular links in the Euclidean  space,  the   singular braid monoid plays an analogous role as that of the braid group  for the classical  links.
 The singular braid monoid  was introduced independently by three authors:
J. Baez, J. Birman and L. Smolin (see  \cite{juJKTR} and the references therein).
\begin{definition}
 The singular braid monoid $SB_n$ is defined as the monoid generated by the  usual braid  generators
 $\sigma_1, \ldots , \sigma_{n-1}$ (invertible)   subject to the following relations:
the  braid relations among the $\sigma_i$'s   together with the following relations:
 $$
 \begin{array}{cccl}
\tau_i \tau_j & = &  \tau_j \tau_i \quad  & \text{for}\quad \vert i-j\vert > 1 \\
  \sigma_i \tau_j & = &  \tau_j \sigma_i \quad  & \text{for}\quad \vert i-j\vert > 1 \\
  \sigma_i \tau_i & = &  \tau_i \sigma_i \quad  & \text{for all} \quad   i  \\
  \sigma_i \sigma_j \tau_i & = &  \tau_j \sigma_i \sigma_j\quad  & \text{for}\quad \vert i-j\vert = 1.
 \end{array}
$$
\end{definition}
Now, in an analogous way to the classical links, we   define the isotopy of the singular links   in the
Euclidean space
in   purely  algebraic terms. More precisely, for the singular links  we have the  analogue  of
the classical Alexander theorem,  which is due to  J. Birman \cite{bi}. We have also the analogue  of the classical
Markov theorem, which is due to  B. Gemein \cite{gemein}. Thus, the set of the isotopy classes of singular knots is in bijection
with the set of equivalence classes defined on the inductive limit associated to the tower of monoids: $SB_1\subseteq SB_2
\subseteq \cdots \subseteq SB_n \subseteq\cdots$ with respect to  the equivalence relation $\sim_s$:
 \begin{enumerate}
 \item[(i)] $\alpha  \beta \sim_s \beta \alpha $
 \item[(ii)] $\alpha\sim_s\alpha \sigma_n $ and  $\alpha\sim_s \alpha \sigma_n^{-1} $
 \end{enumerate}
for all $\alpha, \beta \in SB_n$.

Now we have to define a representation of $SB_n$ in the algebra $\mathcal{E}_n$. This representation uses the same
expression as in \cite{julaJKTR} for its definition. More precisely, we define the representation $\bar{\delta}$ by mapping:
$$
\sigma_i \mapsto T_i \quad \text{and} \qquad \tau_i\mapsto E_i(1 + T_i).
$$
\begin{proposition}
 $\bar{\delta}$ is a representation.
\end{proposition}
\begin{proof}
It is   straightforward   to verify that the images of the defining generators of $SB_n$ satisfy the defining relations of
 $SB_n$.
\end{proof}

In order to define our invariant for singular knots we need to introduce the exponent for the elements of $SB_n$.
From the definition of $SB_n$, it follows that every element  $\omega\in SB_n$  can be written in the
form
$$
\omega = \omega_1^{\epsilon_1}\cdots \omega_m^{\epsilon_m}
$$
where the $\omega_i$ are taken from the defining generators of $SB_n$ and $\epsilon_i= 1$ or $-1$, and assuming moreover
that in the
case $\omega_i$ is  any of  the  generators $\tau_i$,  its  exponent $\epsilon_i$ is by definition equal to $1$. Then  we
  have the following
\begin{definition}\cite[Definition 2]{julaJKTR}
The exponent $\epsilon(\omega)$ of $\omega$ is defined as the sum $\epsilon_1 +\cdots + \epsilon_m.$
\end{definition}
For $\omega \in SB_n$, we define $\bar{\Gamma}$ as follows
$$
\bar{\Gamma}(\omega) =  \left(-\frac{1-\mathsf{L}u}{\sqrt{\mathsf{L}}(1-u)B}\right)^{n-1}
(\sqrt{\mathsf{L}})^{\epsilon(\omega)}
(\rho_n\circ \bar{\delta} )(\omega).
$$

We have then the following theorem.
\begin{theorem}\label{Gamma}
 Let $L$ be a singular link obtained by closing $\omega\in SB_n$; then the
 mapping $L\mapsto \bar{\Gamma}(\omega)$ defines an invariant of singular links.
\end{theorem}
\begin{proof}
The  proof  is totally analogous to the proof of \cite[Theorem 5]{julaJKTR}. See  also the  proof of Theorem \ref{Delta}.
\end{proof}

\subsection{Comparisons}
  In this subsection we shall  show how to obtain    known invariant polynomials  for classical knots  from the three--variable invariant  $\bar{\Delta}$ defined in this paper.

\smallbreak

 In  \cite[Section 6]{joAM} V. Jones constructed  a Homflypt polynomial, denoted $\rm X$, invariant  for  classical links, through   the composition of the Ocneanu trace $\tau_n$, of parameter $z$, on   ${\rm H}_n$ and the representation $\pi_{\lambda}:B_n\rightarrow {\rm H}_n$, $\sigma_i \mapsto \sqrt{\lambda} h_i$, where
$$
\lambda = \frac{z + (1-u)}{uz}.
$$
More precisely, for $\alpha\in B_n$, such Homflypt polynomial is defined by
$$
{\rm X}(\alpha) =\left(-\frac{1-\lambda u}{\sqrt{\lambda}(1-u)} \right)(\tau_n \circ  \pi_{\lambda})(\alpha).
$$
Thus, setting $\mathsf{A}=z$ and  $\mathsf{B}= 1$ in (\ref{definionL}), we obtain $\mathsf{L}= \lambda$. Then, for   $\varphi_n$ of Remark \ref{EontoH}, we have  $\varphi_n\circ \pi_{\mathsf{L}}= \pi_{\lambda}$. Also,  for these values of $\mathsf{A}$ and  $\mathsf{B}$ we  have  $\varphi_n\circ \tau_n= \rho_n$. Then
$$
\tau_n \circ \pi_{\lambda} = \tau_n \circ (\varphi_n\circ \pi_{\mathsf{L}}) = \rho_n \circ \pi_{\mathsf{L}}.
$$
Therefore, it follows that the Homflypt polynomial $\rm X$  can be obtained
   by taking $\mathsf{A}=z$ and specializing $\B=1$.

\smallbreak
Now  we show  how  to  obtain    from  $\bar{\Delta}$ the  two--parameters invariants of  classical  links  defined in \cite{julaMKTA}.

The Yokonuma--Hecke algebra ${\rm Y}_{d,n}$ also supports a Markov  trace, denoted ${\rm tr}$,  of parameters $z$  and $x_1, \ldots , x_{d-1} $, see \cite[Theorem 12]{juJKTR}. In  \cite{julaAM} it is  proved that for certain specific values of the  trace parameters  $x_i$'s it is possible to construct an invariant of classical links $\Delta$. More precisely, these specific values, which are solutions of the so--called  $E$--system,  are  parametrized  by    non--empty subsets of the group of integers modulo $d$. Now, given   such  a subset $S$, we shall denote   ${\rm tr}_S$ the trace ${\rm tr}$, whenever the parameters $x_k$'s are taken as the solutions of the $E$--system,   parametrized by $S$. Now, the mapping  $\sigma_i\mapsto \sqrt{ \lambda_S} g_i$  defines a representation ${\tilde{\pi}}_{\lambda_S}$   of  $B_n$ in ${\rm Y}_{d,n}$, where
$$
\lambda_S = \frac{z + (1-u)/\vert S\vert}{uz}.
$$
The two--variable polynomial invariant of classical knots $\Delta$ is defined  as follows
$$
\Delta (\alpha) =\left(-\frac{1-\lambda_S u}{\sqrt{\lambda_S}(1-u)} \right)({\rm tr}_S \circ  \tilde{\pi}_{\lambda_S})(\alpha)\qquad (\alpha\in B_n)
$$
for details see \cite{julaMKTA}. By taking the parameter $z=\mathsf{A}$ and specializing $B$ to $1/\vert S\vert$, we get that $\lambda_S= \mathsf{L}$. Then, we have $\psi_n\circ \bar{\pi}_{\mathsf{L}}= \tilde{\pi}_{\lambda}$ and ${\rm tr}_S\circ \psi_n = \rho_n$, where $\psi_n$ is  defined in  Proposition  \ref{EtoYH}. Thus,
$$
{\rm tr}_S\circ \tilde{\pi}_{\lambda_S} = {\rm tr}_S\circ (\psi_n\circ \bar{\pi}_{\mathsf{L}})=
   \rho_n\circ \bar{\pi}_{\mathsf{L}}.
$$
Therefore, also the two--variable invariant of classical links $\Delta$ can be obtained
from the three-variable invariant $\bar{\Delta}$.



\section{A diagrammatical interpretation }\label{sectiondiagram}

In this section we recall  a diagrammatical   interpretation of the defining generators of $\mathcal{E}_n(u)$, given in \cite{aijuICTP1}.
 Furthermore,   derived from this diagrammatical interpretation, we  show a   diagrammatic interpretation of the basis
 constructed by S. Ryom--Hansen,  in  which    one part of the elements of this basis   looks  as      {\sl elastic} ties.

Such  interpretation  provides in fact  an  epimorphism  from  the  algebra  $\E_n$ to an algebra of  diagrams, that we call $\widetilde \E_n$,   generated by  the  braid generators,   by  elastic  ties   having  some  peculiar properties explained  below,  and satisfying  the quadratic  relation (\ref{E3}).    We  do  not  furnish here a formal proof  that  $\E_n$  and  $\widetilde \E_n$  are  isomorphic. However,  in what follows,  we will see that this  isomorphism is  crucial,  because the  diagram  calculus  is  considerably  easier  than  the  algebraic one.   Therefore we enunciate the following

\begin{conjecture} The algebras   $\E_n$ and $\widetilde \E_n$  are  isomorphic.
\end{conjecture}

 \subsection{}
 In \cite{aijuICTP1} we have interpreted the  generator $T_i$  as the usual braid  generator   and the generator
 $E_i$ as a tie  between  the   consecutive  strings $i$ and $i+1$.

 \begin{figure}[h]
   \centering
 \includegraphics{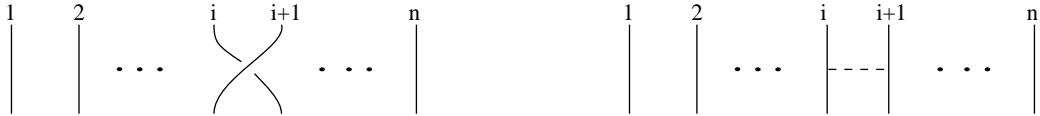}
 \caption{ Generators $T_i$, left,  and $E_i$, right}\label{Fig2}
  \end{figure}

The  relations (\ref{E1})  and (\ref{E2}) of $\E_n$, involving only  the generators $T_i$'s, have  thus  the  usual  interpretation in  terms of  braid  diagrams. The  diagrammatic  interpretation  of the other  relations,  involving  also  the  generators  $E_i$'s, are  shown  in Figures   \ref{btA3} -- \ref{btA5}.

 \begin{figure}[h]
  \centering
  \includegraphics{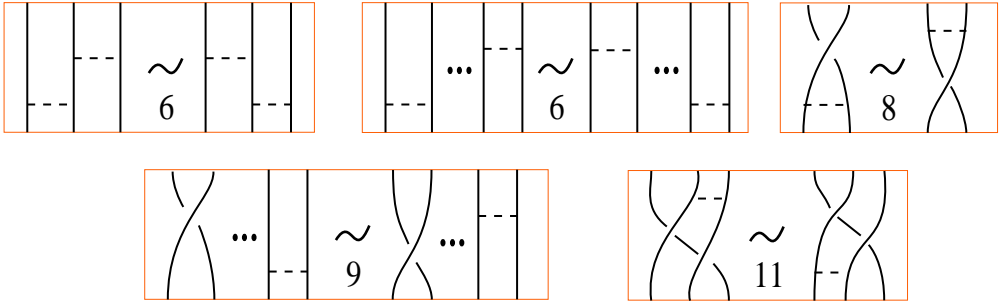}
  \caption{Relations (6),  (8), (9) and (11) in  diagrams}\label{btA3}
   \end{figure}

Observe  that relations (\ref{E4}), (\ref{E6}), (\ref{E7}) and  (\ref{E9}),  depicted  in Figure  \ref{btA3},    simply  indicate   that a  tie  between  two  threads can move  upwards  and downwards along  a  braid as long  as  {\sl such threads  maintain   unit  distance} (we can always  suppose  that  at  each  crossing the  threads maintain their    distance in the three-dimensional space).

On the  other  hand, observe that  this  shifting property of ties  does  not give reason for relation (\ref{E8}), that is shown in  Figure  \ref{btA4}.   Relation  (\ref{E5})  in  Figure  \ref{btA4}   says  that  two  or  more  ties  between  two  threads are equivalent  to  one sole  tie.

 \begin{figure}[h]
  \centering
 \includegraphics{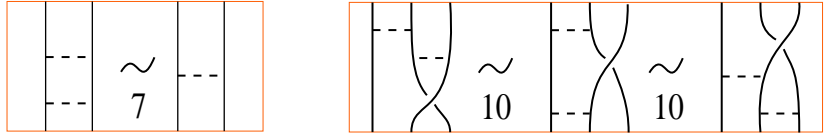}
 \caption{Relations  (7) and (10)  in  diagrams}\label{btA4}
  \end{figure}

Finally, as in the Hecke algebra, the \lq quadratic relation\rq \   (\ref{E3}) takes account of the splitting of the
square of the braid generators in terms of the defining generators. This  relation  is  formally  shown, in terms of diagrams, in Figure \ref{btA5}.

 \begin{figure}[h]
   \centering
  \includegraphics{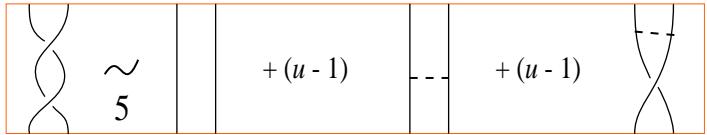}
  \caption{Relation  (5)   in  diagrams}\label{btA5}
   \end{figure}

\begin{remark}\rm \label{rf2} Relations  (\ref{E6})--(\ref{E9})
in $\E_n$ hold  also if  all  generators  $T_i$'s  are  replaced  by  their  inverses.  This  is  verified by  using  formula   (\ref{inverse}) for  the  $T_i^{-1}$s and   relation  (\ref{E5}).  The
 diagrams     of the new   relations  for  the  $T_i^{-1}$'s are  obtained   replacing  the  positive  crossing  by negative  crossings in   the  corresponding diagrams.

However,  substituting  only   $T_i$ (or  only $T_j$)  by  its  inverse in  (\ref{E9}),  we obtain for  instance  the  following
 $$
 E_{i+1}T_i^{-1} T_{i+1} = \quad  T_i ^{-1}  T_{i+1}E_i.
 $$
 This  relation  is  depicted  in figure \ref{btA6}.  We can  thus  observe,  as  we have already  done in  \cite{aijuICTP1}, that  {\sl a  tie  is   allowed to  bypass a thread}.

 \begin{figure}[h]
   \centering
 \includegraphics{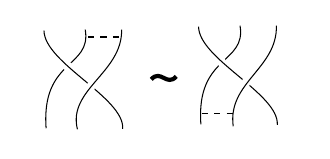}
  \caption{ }\label{btA6}
   \end{figure}

\end{remark}

 \subsection{Elastic  ties}
Recall that  the linear  basis  constructed by Ryom--Hansen (Theorem \ref{basEn}) for $\mathcal{E}_n$  consists of elements of the form
$T_wE_I$, where $w\in S_n$ and $I\in \mathsf{P}_n$. The diagrammatic interpretation   for the elements $T_w$ is standard since the elements $T_i$'s are represented by usual braids.

Remember that the elements $E_I$'s  are  defined by  means  of  the
 $E_{i,j}$'s, where $i<j$, see (\ref{Eij}).  We  introduce now a  simple  diagrammatic  representation  of   the element  $E_{i,j}$, by  means of   an {\sl  elastic  tie} (or {\sl  spring})  connecting the  threads $i$  and  $j$, see  Figure \ref{FEij}.   We  shall say  that  the  spring representing  $E_{i,j}$  has  {\sl  length} $j-i$,  so that  the  element  of  unit  length $E_{i,i+1}$  coincides  with   $E_i$.

 \begin{figure}[h]
   \centering
  \includegraphics{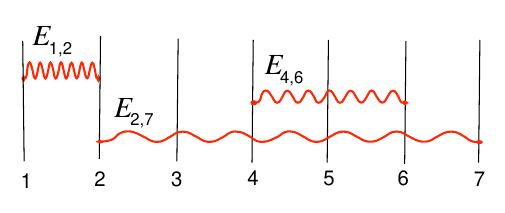}
  \caption{ }\label{FEij}
  \end{figure}

Because  of  the   elastic property  of  the  springs, we   immediately see the  accordance    with  the  original  definition of $E_{i,j}$:
$$
E_{i,j} = T_i \cdots T_{j-2} E_{j-1}T_{j-2}^{-1}\cdots T_{i}^{-1}.
$$
Moreover,  in  Figure \ref{btA8}  we show how $E_{i,j}$  (here,  $E_{2,7}$)  can  be  written  in an equivalent manner by  different   elements  of  the  algebra. (In  fact, any  generator $T_k$ at   left of $E_{j-1}$ may  have  positive  or  negative  unit exponent,  providing  that  at   right of $E_{j-1}$  the same generator  has  the  opposite  exponent).

 \begin{figure}[h]
 \centering
 \includegraphics{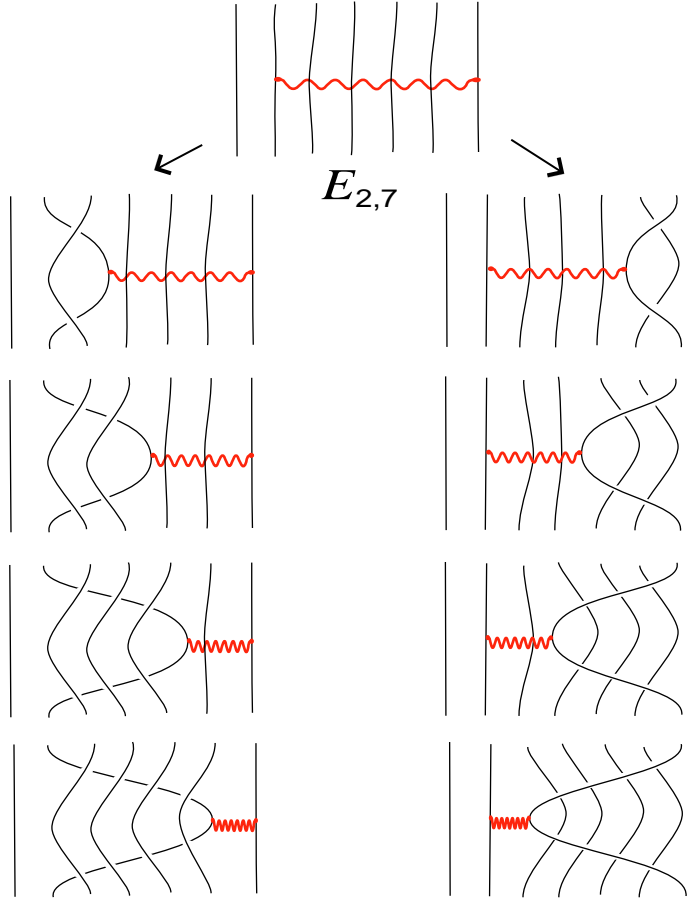}
 \caption{$E_{2,7}= T_2T_3 T_4 T_5  E_6 T_5^{-1}T_4^{-1}T_3^{-1}T_2^{-1} \quad  \sim   \quad  T_6 T_5 T_4 T_3 E_2 T_3^{-1}T_4^{-1}T_5^{-1}T_6^{-1}$.}\label{btA8}
  \end{figure}

\subsection{Properties of  the  elastic  ties} Besides {\it  elasticity}, the  springs  have  some   properties  that  can be  deduced  by  algebraic  calculations  (see  \cite{AJlinks}  for more details  and  proofs).  Here we  show  these properties.

\subsubsection{Transparency} Firstly,    the ties  are  {\sl transparent} for  the  threads,  i.e., they  can  be drawn no  matter  if  in  front  or  behind the  threads. This  can  be observed in  Figure \ref{btA8}.

  Observe  also  that  relation (\ref{E9}),  as  well  as  Remark \ref{rf2},  have  a  generalization  for  springs of  any  length,   as  shown  in   Figure \ref{btA12}   (case of  length equal  to  2).

 \begin{figure}[H]
 \centering
  \includegraphics* {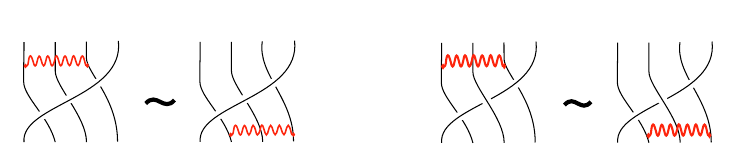}
 \caption{ }\label{btA12}
 \end{figure}

 \subsubsection{Transitivity}  The  product of  three  springs     $E_{i,j}$  $E_{j,k}$ and  $E_{i,k}$  connecting the  threads $i$, $j$,  and  $k$, is  equivalent  to  the product of  any two  of the springs. So,  in particular    $E_{i,j}E_{i,k}$  is  equivalent  to  the  product $E_{i,j}E_{j,k}$.
 Note  that this  property   implies  the  equivalence  of  formulae  (\ref{EJ})  and  (\ref{EJI}).    We  shall show   two   cases for $n=7$.

Set
$$
I_1:=( \{2,3,5,7  \},\{1,4,6\}),  \quad I_2:=(\{2,3,5,6,7  \},\{1\}, \{4 \}).
$$
Then $E_{I_1}$ and $E_{I_2}$ have the   diagrams shown in Figure \ref{btA9},  according  to  (\ref{EJ}).
 \begin{figure}[H]
 \centering
 \includegraphics{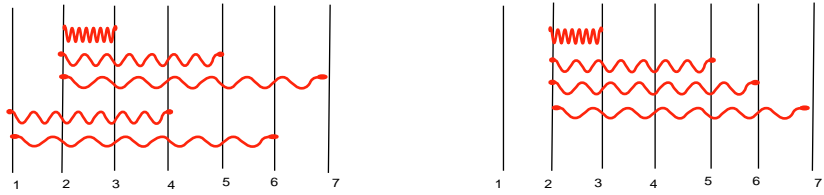}
  \caption{$E_{I_1}=(E_{2,3}E_{2,5}E_{2,7})(E_{1,4}E_{1,6}) \quad \quad
  \quad \quad   E_{I_2}=(E_{2,3}E_{2,5}E_{2,6}E_{2,7})$}
 \label{btA9}
  \end{figure}

These   elements  can  be  represented  by  the   diagrams drawn  in  Figure  \ref{btA10}.
 \begin{figure}[h]
  \centering
  \includegraphics{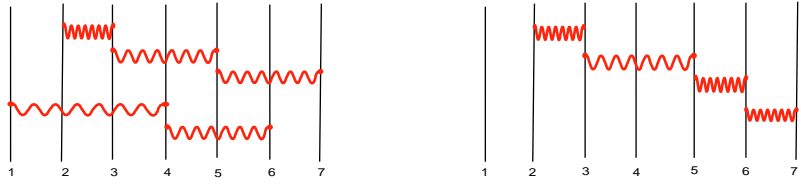}
   \caption{$E_{I_1}=(E_{2,3}E_{3,5}E_{5,7})(E_{1,4}E_{4,6}) \quad \quad
  \quad \quad   E_{I_2}=(E_{2,3}E_{3,5}E_{5,6}E_{6,7})$}\label{btA10}
 \end{figure}

\subsubsection{Mobility}
Let $s=\pm1$. Observe  the identities
\begin{equation}\label{commut1}
  T_i^s   E_{i+1}=  T_i^s  E_{i+1} ( T_i^{-s} T_i^{s})=  (T_i^s  E_{i+1}   T_i^{-s})  T_i^{s} = E_{i,i+2} T_i^s,
\end{equation}
\begin{equation}\label{commut2}
  E_{i+1} T_{i}^s = ( T_{i}^{s} T_{i}^{-s})   E_{i+1} T_{i}^s=  T_{i}^{s} (T_{i}^{-s}    E_{i+1} T_{i}^s)= T_{i}^s E_{i,i+2}.
\end{equation}
They can  be  interpreted  as  the  sliding  of  the  tie  up  and  down along  the  braid  under  stretching or  contracting.
In  other  words,  while  the  element  $T_i^{\pm 1}$     does  not  commute  with    $E_j$   when  $|j-i|=1$,  equations  (\ref{commut1}) and    (\ref{commut2}) provide   a  sort  of  commutation  rule between   $T_i^{\pm 1}$    and  the elastic tie.  See  in  Figure \ref{btA11}  the  sliding down (by  contracting)  of the  red  spring   and  the sliding down of the green spring (by  extending).

Similarly,  a spring $E_{i,j}$ of  any  length bigger  than one, \lq commutes\rq\   (changing  its  length  by $\pm1$)   with   $T_{i}$   and $T_{i-1}$,  as  well as with  $T_{j-1}$  and  $T_j$,  according  to  the    equalities:
$$  E_{i,j}T_{i}= T_{i} E_{i+1,j},  \quad  E_{i,j} T_{i-1} = T_{i-1}  E_{i-1,j}, \quad E_{i,j}T_{j}=T_{j}E_{i,j+1}, \quad E_{i,j}T_{j-1}T_{j-1}= E_{i,j-1}.$$
 The  same  equalities  hold  for the  inverse of  the  generators $T_i$'s.

\begin{remark} \rm The  peculiar relation  (\ref{E8}),   see  Figure \ref{btA4}, plays an  essential  role, together with  relation (\ref{E5}),  in  the  proof   of the transitivity property.  Using  this  property  and  the mobility property,  relation (\ref{E8})    in  terms  of springs  becomes clear,  as Figure \ref{btA11}  shows.

 \begin{figure}[H]
  \centering
  \includegraphics{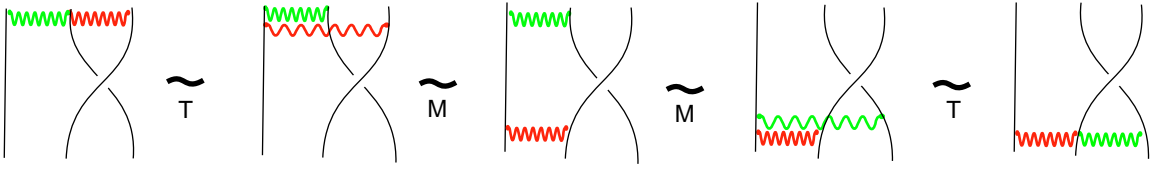}
  \caption{$\T$  and $\M$  indicate  the  Transitivity  and  Mobility properties  }\label{btA11}
 \end{figure}
\end{remark}

\section{Side Comments}
We  conclude with two comments which we think   deserve  to be examined in  depth.
\subsection{}
The referee has suggested the following: it would be interesting to know whether there is an integrable model based on the bt--algebra and built  with the use of relative traces.
According to the referee report, a good indication of the existence of an
integrable model related to the bt--algera is the  fact that  the relative
trace has  the following  property.  For all $X\in {\mathcal E}_{n-1}$, we have:
$$
 \varrho_{n}(T_{n-1}^{-1}X T_{n-1})= \varrho_{n-1}(X),  \qquad
 \varrho_{n}(T_{n-1}X T_{n-1}^{-1})= \varrho_{n-1}(X).
$$
\subsection{}
    In Subsection \ref{explanE}  it was noted that the bt--algebra  is conceived from  the Yokonuma--Hecke algebra. Now, the Yokonuma--Hecke algebra can be regarded as the prototype  example of the framization of a knot algebra, see \cite{julaWS2} for the  concept of framization and knot algebra. In few words,  the Yokonuma--Hecke algebra can be considered  as the Hecke algebra to  which  one  adds  framing generators:
the defining generators of the Yokonuma--Hecke algebra consist in fact in a set of braid generators and a set of framing generators.

As we explained is  Subsection \ref{explanE}, the construction of the bt--algebra is done  by considering abstractly the algebra  generated by the braid generators  $g_i$'s  together  with the idempotents $e_i$'s. Despite the  fact that  the $e_i$'s are  defined  by  means of the  framing generators (see formula (\ref{eisumti})),  such  generators  do not appear in the bt--algebra. So, starting from the Hecke algebra, the Yokonuma--Hecke algebra has  been constructed    by adding framing generators. Now, in the opposite direction, we  constructed from the Yokonuma--Hecke the bt--algebra, not containing framing generators. For this reason we   say that the bt--algebra is a \lq deframization\rq\ of the Yokonuma--Hecke algebra.

  Thinking in this way one can   define   naturally  deframizations of   all the algebras of knots  framized listed in \cite{julaWS2}. Moreover, there is a natural deframization associated to certain algebras ${\rm Y}(d,m,n)$ defined in  \cite{chpoArxiv}, where $d$, $n$ are positive integers and $m$ is either a positive integer or $\infty$. More  precisely,  for any  positive integer $a$, set  $v_1,\dots , v_a$  as indeterminates. Set $K_m := K( v_1, \ldots, v_m)$  for a positive integer $m$   and $K_{\infty}:= K$;  we  can  define  a deframization of
 ${\rm Y}(d,m,n)$  as the associative algebra over $K_m$ generated by $T_1,\ldots , T_{n-1}$, $E_1, \ldots, E_{n-1}$, $X^{\pm 1}$ subject to the  relations (\ref{E1}) to (\ref{E9}) together with the following relations:
 $$
\begin{array}{rcll}
XT_1XT_1 & = & T_1 X T_1 X  &\\
X T_i  & = & T_i X & \text{for }\quad i\in \{2, \ldots , n-1\}\\
XE_i & = & E_i X & \text{for }\quad i\in \{1, \ldots , n-1\}\\
(X-v_1) \ldots (X- v_m) & =  & 0 & \text{for}\quad m<\infty.
\end{array}
$$
It is worth to note that  the algebras ${\rm Y}(d,m,n)$,   can be
regarded as framizations of knot algebras,  and  that in fact  ${\rm Y}(d,1,n)$  is    the  Yokonuma--Hecke  algebra. Notice that by applying the deframization   to a  framized algebra,
 we do  not recover the original algebra.

\vskip 1 cm

   \noindent {\it Added, October 28, 2015}:
As we explained in Subsection 5.3,  the invariant $\bar{\Delta}$ is a generalization of the invariant  $\Delta$ for classical links,  defined in \cite{julaAM, julaMKTA}. These invariants are constructed by using the Jones' recipe applied to the bt--algebra and the Yokonuma--Hecke algebra, respectively; in both algebras  a similar expression for the quadratic relation is  used. In \cite{chljukala},   an invariant $\Theta$  for  classical  links   is  defined, starting  from  the Yokonuma--Hecke algebra,  but  using a different  quadratic  relation  with  respect to that used  in  the definition of  $\Delta$.
 In  the  same paper \cite{chljukala}, it is proved  that  $\Theta$ coincides with the Homflypt polynomial  on knots   (this was proved later also in \cite{japo}), but it may distinguish links that are not distinguished by the Homflypt polynomial. Recently, the first author of the present paper, has  verified  that  $\bar \Delta$ distinguishes  pairs  of  links  not  distinguished  by  the  Homflypt polynomial,  and  that the pairs distinguished by $\bar\Delta$ and  $\Theta$  do not  coincide in  some  cases.  Therefore, we know that    the invariant $\bar{\Delta}$ for links
is more powerful  that the Homflypt polynomial,  but  its  relation  with  $\Theta$  deserves  a  deeper investigation.

\end{document}